\documentclass{amsart}

\usepackage[applemac]{inputenc}
\usepackage{amsmath, amssymb}
\usepackage{bbm}
\usepackage{latexsym}
\usepackage[dvipsnames]{xcolor}
\usepackage{tikz}
\usetikzlibrary{arrows}
\usetikzlibrary{patterns}
\usetikzlibrary{shapes}
\usepackage{mathrsfs}
\usepackage{mathtools}
\usepackage{verbatim}
\usepackage{bbm}
\usepackage{hyperref}
\hypersetup{
	colorlinks,
	linkcolor={red!50!black},
	citecolor={blue!50!black},
	urlcolor={blue!80!black}
}
\usepackage[mathcal]{euscript}

\newcommand{\R}{\mathbb{R}}
\newcommand{\N}{\mathbb{N}}

\newcommand{\eexp}{\overrightarrow{\exp}}
\newcommand{\eexpi}{\overleftarrow{\exp}}

\newcommand{\ind}{\mathrm{ind}}

\newcommand{\Vc}{\mathrm{Vec}}
\newcommand{\He}{\mathrm{He}}
\newcommand{\ov}{\overline}
\newcommand{\mc}{\mathcal}
\newcommand{\IM}{\mathrm{Im} \,}
\newcommand{\proj}{\pi}

\newcommand{\BS}{L^\infty(I,\R)\oplus L^2(I,\R)}

\usepackage[usenames,dvipsnames]{pstricks}
\usepackage{epsfig}
\usepackage{pst-grad} 
\usepackage{pst-plot} 

\newtheorem{thm}{Theorem}
\newtheorem{lemma}[thm]{Lemma}
\newtheorem{cor}[thm]{Corollary}
\newtheorem{prop}[thm]{Proposition}

\theoremstyle{definition}

\newtheorem{defi}[thm]{Definition}

\theoremstyle{remark}
\newtheorem{remark}[thm]{Remark}

\newcommand{\be}{\begin{equation}}
\newcommand{\ee}{\end{equation}}

\usepackage{mathtools} 
\mathtoolsset{showonlyrefs,showmanualtags} 

\numberwithin{equation}{section}

\title[Normal forms for the endpoint map near nice singular curves]{Normal forms for the endpoint map near nice singular curves}
\author{Andrei A. Agrachev}
\address{SISSA (Trieste) and Steklov Institute (Moscow)}
\email{\href{agrachev@sissa.it}{\nolinkurl{agrachev@sissa.it}}}

\author{Francesco Boarotto}
\address{Dipartimento di Matematica Tullio Levi-Civita,
 Universit\`a degli studi di Padova, Italy}
\email{\href{mailto:francesco.boarotto@math.unipd.it}{\nolinkurl{francesco.boarotto@math.unipd.it}}}



\thanks{F.B. has been supported by 
the ANR SRGI (reference ANR-15-CE40-0018), and by
the University of Padova STARS Project ``Sub-Riemannian Geometry and Geometric Measure Theory Issues: Old and
	New"}

\begin{document}
	
	\maketitle
	
	\begin{abstract}
	Given a rank-two sub-Riemannian structure $(M,\Delta)$ and a point $x_0\in M$, a singular curve is a critical point of the endpoint map $F:\gamma\mapsto\gamma(1)$ defined on the space of horizontal curves starting at $x_0$. The typical least degenerate singular curves of these structures are called \emph{regular singular curves}; they are \emph{nice} if their endpoint is not conjugate along $\gamma$.
	The main goal of this paper is to show that locally around a nice singular curve $\gamma$, once we choose a suitable topology on the control space we can find a normal form for the endpoint map, in which $F$ writes as a sum of a linear map and a quadratic form. We also study the restriction of $F$ to the level sets of the action functional and give a Morse-like formula for the inertia index of its Hessian at $\gamma$. This is a preparation for a forthcoming generalization of the Morse theory to rank-two sub-Riemannian structures.
	\end{abstract}
	
	\section{Introduction}\label{sec:Introduction}
	\subsection{Horizontal path spaces and singular curves}Let $M$ be a smooth manifold of dimension $m>2$ and consider a smooth (i.e. $C^\infty$-smooth), totally nonholonomic distribution $\Delta\subset TM$ of rank $2$. Define $I:=[0,1]$.
Given  a point $x_0\in M$ (which we will assume fixed once and for all) the horizontal path space $\Omega$ of \emph{admissible} (also called \emph{horizontal}) curves starting at $x_0$ is defined by:
	\be\label{eq:Omega}
		\Omega=\{\gamma:I\to M\,|\, \gamma(0)=x_0, \,\textrm{$\gamma$ is absolutely continuous, $\dot{\gamma}\in \Delta$ a.e. and is $L^2$-integrable}\}.
	\ee
	The $W^{1,2}$-topology endows $\Omega$ with a Hilbert manifold structure, locally modeled on $L^2(I, \R^2)$. The \emph{endpoint map} $F:\Omega\to M$ is the smooth map assigning to each curve its final point  $F(\gamma)=\gamma(1)$. Given $y\in M$ we call $\Omega(y):=F^{-1}(y)$ the set of all horizontal curves joining $x_0$ and $y$. If $y$ is a \emph{regular} value of $F$, then $\Omega(y)$ is a smooth Hilbert submanifold. In general, however, $y$ is not regular and $\Omega(y)$ is a submanifold with singularities.
	
	A curve $\gamma$ is \emph{singular} (or \emph{abnormal}) if $d_\gamma F:T_\gamma \Omega\to T_{F(\gamma)}M$ is not surjective. The \emph{corank} of $\gamma$ is the codimension of $\IM(d_\gamma F)$ in $T_{F(\gamma)}M$. Singular curves are central objects in the theory of nonholonomic distributions, but their study is a difficult problem and many fundamental questions related to their structure are still open \cite{Mont02,Ag_Open}. Most of the difficulties come from the fact that the Hessian of $F$ is always degenerate, and even in the simplest case of a singular curve $\gamma$ of corank one, it is difficult to obtain local informations on $F$ from the Hessian: for example it is known \cite{boarotto2015homotopy} that the endpoint map is always locally open in the $W^{1,2}$-topology even if the Hessian has a sign. 
	
	Nevertheless, singular curves are interesting for the following reason. Fix on $\Delta$ a Riemannian metric, and denote by $\|\cdot\|_{L^2(I,\R^2)}$ the associated $L^2$ norm. Define then the \emph{energy} functional $J:\Omega\to \R$ by the formula $J(\gamma)=\frac{1}{2}\|\dot{\gamma}\|_{L^2(I,\R^2)}$. The \emph{sub-Riemannian minimizing problem} between $x_0$ and $y\in M$ consists into finding the admissible curves realizing $\min\{ J(\gamma)\mid \gamma\in\Omega(y)\}$, and Montgomery proved for the first time in \cite{AbnMin} that singular curves may be solutions to the problem (in fact, even independently on the choice of the metric).
	
	\subsection{Rank-two-nice singular curves} Let $\mc{F}:\Omega\to M\times \R$ denote the \emph{extended} endpoint map, that is $\mc{F}$ is the pair $(F,J)$. Candidate solutions $\gamma$ to the sub-Riemannian minimizing problem, called \emph{extremals}, are constrained extremal points for $\mc{F}$ and in particular a singular curve is an extremal. An extremal is called \emph{strictly singular} (or \emph{strictly abnormal}) if, for every nonzero $\xi\in T^*_{\mc{F}(\gamma)}(M\times \R)= T^*_{F(\gamma)}M\times \R$ such that $\xi d_\gamma \mc{F}=0$, the projection of $\xi$ onto its $\R$-factor is zero.
	
	Let us define $\Delta^2:=[\Delta,\Delta]$ and $\Delta^3:=[\Delta,\Delta^2]$, where $[\cdot,\cdot]$ are the usual Lie brackets. Borrowing the terminology from \cite{SLShortest}, we say that a
	singular curve $\gamma$ is \emph{regular}
	if the Pontryagin Maximum Principle provides a nowhere zero curve $\eta(t)$ (sometimes called \emph{biextremal} in the literature), dual to $\gamma(t)$, such that $\eta(t)\in (\Delta_{\gamma(t)}^2)^\perp\setminus (\Delta_{\gamma(t)}^3)^\perp$ for every $t\in [0,1]$.
	Regular singular curves are the least degenerate singular curves for sub-Riemannian structures of rank $2$. On the contrary, for the generic (with respect to the $C^\infty$-Whitney topology on the set of distributions $\Delta$) sub-Riemannian structure of rank greater than or equal to $3$ there are no regular singular curves \cite[Corollary 2.5]{CJT06}. A regular singular curve $\gamma$ is smooth \cite[Theorem 4.4]{AS96}.

	\begin{defi}\label{defi:ranktwonice}
		An admissible singular curve $\gamma$ is a \emph{rank-two-nice} singular curve if
		\begin{itemize}
			\item[(i)] $\gamma$ is a corank-one regular strictly singular curve.
			\item[(ii)] $y=F(\gamma)$ is not a conjugate point along $\gamma$.
		\end{itemize}
	\end{defi}
	
	Having already explained the meaning of (i), we briefly discuss (ii). Consider the subspace $\ker(d_\gamma F)\subset T_\gamma\Omega$, equipped with the $W^{1,2}$-topology. In particular, if $\gamma$ has corank one, $\ker(d_\gamma F)$ is of codimension $m-1$ in $T_\gamma\Omega$ and the Hessian of $F$ at $\gamma$ is a
	$T_{F(\gamma)}^*M/\IM (d_\gamma F)$-valued quadratic form on $\ker(d_0F)$. Notice that $T_{F(\gamma)}^*M/\IM (d_\gamma F)\simeq \R$.

It turns out that we may consider on $\Omega$ a topology which is weaker than the $W^{1,2}$-topology, for which the endpoint map $F$ is still well-defined. We denote by $\mathring{\Omega}$ the closure of $\Omega$ with respect to this topology,
	and by $\mathring{F}$ the continuous extension of $F$ to $\mathring{\Omega}$.
	We say that $y=F(\gamma)$ is not a conjugate point along $\gamma$ if, essentially, $\ker(\He_\gamma \mathring{F})\subset\ker(d_\gamma \mathring{F})$ is ``minimal''.  
	The precise definitions, which are rather technical, are given with all the details in Section~\ref{sec:conjpoints} and specifically in Definition~\ref{defi:conj}. We only remark that, once a regular strictly singular curve $\gamma$ is chosen, the set of $s$ such that $y=\gamma(s)$ is not conjugate along $\gamma$, i.e. the set of times $s$ such that $\gamma|_{[0,s]}$ is rank-two-nice, is dense in $I$ by \cite[Lemma 7]{Sarysec}.

	\subsection{Local coordinates and the main theorem}
	
	 Let $\gamma\in\Omega(y)$ be a rank-two-nice singular curve. Our study of the endpoint map $F$ being local around $\gamma$,
we assume without loss of generality that $\gamma$ does not have self-intersections, lifting if necessary both $\Delta$ and $\gamma$ to a covering space of a neighborhood $\mc{O}_{\text{supp}(\gamma)}\subset M$ of $\text{supp}(\gamma):=\{\gamma(t)\mid t\in I\}$.
	
	If $\gamma$ does not have self-intersections,
there exists a pair of smooth vector fields $X_1$ and $X_2$ such that
	\be\label{eq:loc_param}
		\Delta_x=\mathrm{span}\left\{ X_1(x),X_2(x) \right\}
	\ee
	for every $x\in \mc{O}_{\text{supp}(\gamma)}$, and such that $\gamma$ is an integral curve of the field $X_1$.

We parametrize admissible curves $\xi\in\Omega$ contained in $\mc{O}_{\text{supp}(\gamma)}$ as integral curves on $M$ of the differential system:
	\be\label{eq:ctrlsystem}
		\dot{\xi(t)}=(1+v_1(t))X_1(\xi(t))+v_2(t)X_2(\xi(t)),\ \ \xi(0)=x_0,\ \ \text{a.e. }t\in I,
	\ee
	with $v=(v_1,v_2)\in L^2(I,\R^2)$. Identifying an admissible curve $\xi$ with its control $v$ via \eqref{eq:ctrlsystem}, by a slight abuse of notation we
set
	\be
		F(v)=F(v_1,v_2)=F(\xi).
	\ee
In particular $F(0)=F(\gamma)$.

The main theorem of our paper gives a local normal form of $F$. Local in this setting does not just mean: ``in a neighborhood of zero in $L^2(I,\R^2)$'', but is a bit more delicate. Given a subspace $E\subset L^2(I,\R^2)$, the intersection of $E$ with a neighborhood of the origin in $L^\infty(I,\R)\oplus L^2(I,\R)$ will be called an $(\infty,2)$-neighborhood of the origin in $E$. For every $d\in\N$, an $(\infty,2)$-neighborhood of the origin in $E\oplus\R^d$ is the sum of an $(\infty,2)$-neighborhood of the origin in $E$ and
a neighborhood of the origin in $\R^d$.			
	\begin{thm}\label{thm:main}
	Let $\gamma\in\Omega(y)$ be a rank-two-nice singular curve.
Then there exist an origin-preserving homeomorphism $\mu:\mc{V}\to\mc{V}'$, of $(\infty,2)$-neighborhoods of the origin $\mc{V}\subset \ker(d_0F)\oplus \IM (d_0F)$ and $\mc{V}'\subset L^2(I,\R^2)$,
and a diffeomorphism $\psi:\mc{O}_y\to \mc{O}_0$ of neighborhoods $\mc{O}_y\subset M$ and $\mc{O}_0\subset \R\oplus \IM (d_0F)$, respectively of $y$ and $0$, such that
	\be
		\psi\circ F\circ \mu(v,w)=\left( \He_0F(v),w \right)
	\ee
	for every $(v,w)\in \mc{V}$.
	\end{thm}

\begin{remark} The class of available $(\infty,2)$-neighborhoods does not depend on a particular choice of the frame as long as $\gamma$ is an integral curve of $X_1$, since a change of the frame would result in a smooth change of local coordinates in the space of horizontal curves.\end{remark}
	
	 The restriction to $(\infty,2)$-neighborhoods is not by chance, and in fact Theorem~\ref{thm:main} cannot be true
if we put the $L^2(I,\R^2)$-topology on the space of controls. If $\gamma$ is a rank-two-nice curve, the negative eigenspace $N\subset L^2(I,\R^2)$ of the Hessian $\He_0F$ is finite dimensional, and possibly empty (see Proposition~\ref{prop:ConjProp}). Moreover, the restriction of $F$ to any subspace of finite codimension is an open map \cite[Proposition 2]{boarotto2015homotopy}.
	Assuming Theorem~\ref{thm:main} to hold in $L^2(I,\R^2)$, on the finite codimensional space $N^\perp:=\{x\in L^2(I,\R^2)\mid (x,y)_{L^2(I,\R^2)}=0,\, \text{for every }y\in N\}$ the projection of $F$ along the ``abnormal'' direction would instead have the sign of $\He_0F$, yielding an absurd.
	
	  We close the paragraph recalling that, if there are no conjugate points along $\gamma$ between $x_0$ and $y=\gamma(1)$, it is known \cite{AS96} that $\gamma$ is a local minimizer in the $L^2$-topology on the space of controls. This result, however, is proved by showing that the perturbation provided by $\He_0F$ along the abnormal direction is not compensated by higher order terms, and does not require any normal form for $F$ (which does not exist!).
	
	  \subsection{Connections with Morse Theory} Theorem~\ref{thm:main} is a ``nonstandard'' infinite dimensional Morse Lemma. It is nonstandard because the first map $\mu$ is just an homeomorphism, and not a diffeomorphism as in the classical setting. This difference is intrinsic to the geometry of the problem and is due to the fact that $\He_0F$ is highly degenerate.
	
	 To give a flavour of the results that descend from Theorem~\ref{thm:main}, we state the following result, whose proof is given in Section~\ref{sec:isolation}.
	\begin{thm}\label{thm:isolation}
		There exists a neighborhood of the origin $\mc{U}\subset \BS$  such that, if $u\subset \mc{U}$ is a control associated with an extremal curve in $\Omega(y)$, then $u=(u_1,0)$ a.e. on $I$, and $\int_0^1u(t)dt=0$.
	\end{thm}
	
	 The topology we require in Theorem~\ref{thm:isolation} is stronger than the $W^{1,2}$-topology, hence our result does not allow to conclude about the local minimality of $\gamma$ in $\Omega$. Morover, denoting by $\ind(q)$ the inertia index of a given quadratic form $q$, we know that $\gamma$ is not a local minimum for $J$ as soon as $\ind(\He_0F)\ge 1$.

	 At any rate, with respect to the $\BS$ topology, Theorem~\ref{thm:main} reduces the study of $\Omega(y)$ to the study of the infinite-dimensional quadratic cone $\left\{\He_{0}F=0\right\}$. If $\ind(\He_{0}F)=0$, in this topology $\gamma$ is isolated in $\Omega(y)$ and we recover the \emph{rigidity} results of \cite{AS96}, which generalize those in \cite{BryantHsu}. This is true independently on the functional we are trying to minimize.

	To bring also the energy $J$ into the picture, and to study the sub-Riemannian minimizing problem more closely, we need to consider the extended endpoint map $\mc{F}$. In fact, since $\gamma$ is a corank-one strictly abnormal curve, $\He_0 F$ and $\He_0 \mc{F}$ have the same expression, but their domains are different because $\ker(d_0 \mc{F})=\ker(d_0 F)\cap \ker (d_0 J)$ has codimension one in $\ker(d_0 F)$.
	
	If $y=\gamma(1)$ is not conjugate along $\gamma$ neither for $\mc{F}$, then $\mc{F}$ has a normal form as well, and Theorem~\ref{thm:main} makes evident that
	the topological information provided by a rank-two-nice curve $\gamma$ is contained:
	 \begin{itemize}
	 	\item  in the pair $\left(\ind(\He_0F),\ind(\He_0\mc{F})\right)$ (by construction we have $\ind(\He_0\mc{F})\le \ind(\He_0F)\le \ind(\He_0\mc{F})+1$),
		\item  in the relative position of $\mathcal{J}:=\proj_{\ker(d_0 F)}\nabla_0 J$ with respect to the cone $\{\He_0F=0\}$, as shown in Figure~\ref{fig:norfor}.
	 \end{itemize}
	
	\begin{figure}
			\includegraphics[scale=.5]{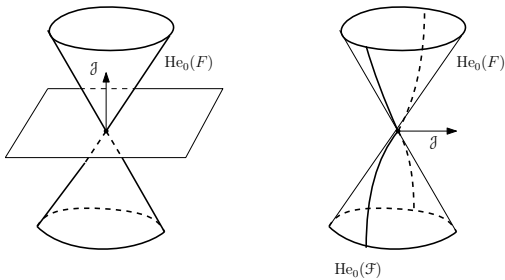}
			\caption{The relative positions between $\mathcal{J}:=\proj_{\ker(d_0 F)}\nabla_0 J$ and the quadratic cone $\{\He_0F=0\}$. In the first case
$\He_0F$ and $\He_0\mc{F}$ have different indexes, in the second
the indexes are the same.}
			\label{fig:norfor}
		\end{figure}
	
	For example, if $(\ind(\He_0F),\ind(\He_0\mc{F}))=(1,0)$, it is no longer true that $\gamma$ is isolated on the whole $\Omega(y)$, but only on the intersection of $\Omega(y)$ with the level set $\{J=J(\gamma)\}$. More generally, the knowledge of $\left(\ind(\He_0F),\ind(\He_0\mc{F})\right)$ allows to classify the singularity of $\Omega(y)$ around $\gamma$.
	
	We expect that rank-two-nice curves influence the homotopy type of the Lebesgue sets of $J$, i.e. of the sets $\{J\le c\} $ for $c\in \R$, and we plan to study how the ``homotopical  (in)visibility'' (in the $W^{1,2}$-topology) of a rank-two-nice curve $\gamma$ is affected by the interplay between $\ind(\He_0F)$ and $\ind(\He_0\mc{F})$ in a forthcoming paper, completing the Morse-like results contained in \cite{ABL17} to the case of sub-Riemannian structures of rank two. 		
	
	\subsection{An explicit computation of conjugate times}\label{sec:example} Let $M=SO(3)\times \R$ and $\mathfrak{m}=\mathfrak{so}(3)\oplus\R$ be its Lie algebra. Let
	\be
		X_1:=\frac{(T_1+T_2)\oplus 2}{\sqrt{2}},\ \  X_2:=\frac{T_1\oplus 1}{\sqrt{2}},
	\ee
 	where $T_1,T_2,T_3$ are the standard generators for $\mathfrak{so}(3)$, that is $[T_1,T_2]=T_3$, $[T_2,T_3]=T_1$, $[T_3,T_1]=T_2$.
	
	We define a distribution $\Delta\subset TM$ extending these vectors to vector fields on $M$ by left-translation
		and we consider on $\Delta$ the metric that makes $X_1$ and $X_2$ orthonormal.
		The energy of a horizontal curve is
	the squared $L^2$-norm of its velocity.
	
		We consider $\gamma$ to be an integral curve of $X_1$.
		This curve satisfies point (i) in Definition~\ref{defi:ranktwonice}
by the results in \cite[Section 8]{Sus4d}. We denote by $\gamma_s:=\gamma\big|_{[0,s]}$ the restriction of $\gamma$ to the interval $[0,s]$.
	The set of times $s\in I$ for which $\gamma(s)$ is a conjugate point along $\gamma$ for the map $F$ (resp. for the map $\mc{F}$) is given
	by the sets $\{a_F=0\}$ (resp. $\{a_{\mc{F}}=0\}$), where:
		\be
			\begin{aligned}
		a_F(s)&=
		\sin(s),\\
		a_{\mc{F}}(s)&=s\sin(s)+2(\cos(s)-1).
			\end{aligned}
		\ee
 We will explain later at the end of Section~\ref{sec:fourdim} how to derive such equations. If we pick a point $s_0$ such that $a_F(s_0)\ne 0\ne a_{\mc{F}}(s_0)$, then there exists a normal form for both $F$ and $\mc{F}$.

		The mapping $s\mapsto \left(\ind (\He_{\gamma_s}F), \ind(\He_{\gamma_s}\mc{F})\right)$ is piecewise constant on $I$:  the value of $\ind (\He_{\gamma_s}F)$ (resp. of $\ind (\He_{\gamma_s}\mc{F})$) changes when $s$ is a zero of $a_{F}$ (resp. a zero of $a_{\mc{F}}$).  In particular, there exist time intervals arbitrarily far from zero on which either $\He_{\gamma_s}F$ and $\He_{\gamma_s}\mc{F}$ have the same index, or where these indexes differ by one. Looking at Figure~\ref{fig:plot}, we see that the initial values of $\left(\ind (\He_{\gamma_s}F), \ind(\He_{\gamma_s}\mc{F})\right)$ are given by:
		\be
			(0,0),(1, 0),(2,1),(2,2),(3,2),(4,3),(4,4),(5,4),(6,5),\ldots
		\ee

	\begin{figure}
		\includegraphics[scale=.6]{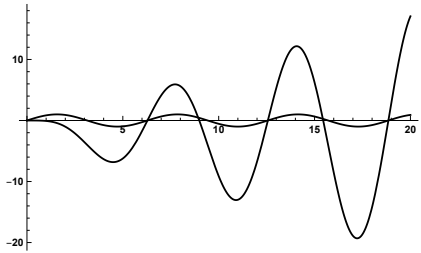}
		\caption{Plot of the functions $a_F(s)$ and $a_{\mc{F}}(s)$}
		\label{fig:plot}
	\end{figure}
	
	\subsection{Structure of the paper} We give in Section~\ref{se:preliminaries} all the technical preliminaries, and we study the endpoint map $F$ at the first and at the second order. We discuss conjugate points along a rank-two-nice singular curve $\gamma$ in Section~\ref{sec:conjpoints}, where we also establish some useful analytical properties of the Hessian map.
	In Section~\ref{sec:deformation} we construct the reparametrization map which is needed to bypass the kernel of the Hessian map, even though this requires the passage to the space $\BS$. Section~\ref{se:norfor} is the core of the paper, where we prove the existence of a normal form for the endpoint map $F$ locally around $\gamma$ and we give the proofs of Theorems~\ref{thm:main} and ~\ref{thm:isolation}. Finally, in Section~\ref{sec:examples} we characterize conjugate points along $\gamma$ on the cotangent space $T^*M$ in terms of an appropriate Jacobi equation, and we study in details the class of Engel structures, completing the explanation of our previous example.

	\section{Rank-two sub-Riemannian structures}\label{se:preliminaries}
	\subsection{
	Technical preliminaries}\label{sec:SC} Let $\Omega$ be defined as in \eqref{eq:Omega}.
	For $t\in I$, the map $F_t$ is the map that returns the point $\gamma(t)$ of a horizontal curve $\gamma\in \Omega$:
	\be
		 F_t:\Omega\to M,\quad F_t(\gamma)=\gamma(t).
	\ee
	Then $F=F_1$ is the endpoint map. We recall in the following proposition some useful properties of $F_t$ (see \cite{Tre00,boarotto2015homotopy}).
	\begin{prop}\label{prop:contendp}
		For every $t\in I$, the map $F_t:\Omega\to M$ is smooth with respect to the Hilbert manifold structure on $\Omega$. Moreover, if $\gamma_n\rightharpoonup\gamma$ weakly, then $F_t(\gamma_n)\to F_t(\gamma)$
	and $d_{\gamma_n}F_t\to d_{\gamma}F_t$ in the operator norm, uniformly with respect to $t$.
	\end{prop}

	\begin{defi}\label{defi:sing}
		We say that $\gamma\in\Omega$ is a singular (or abnormal) curve if $\gamma$ is a critical point of $F$, or equivalently if $d_{\gamma}F:\Omega\to T_{F(\gamma)}M$ is not surjective. The corank of $\gamma$ is
	the codimension of the image $\IM(d_\gamma F)$ of $d_{\gamma}F$ in $T_{F(\gamma)}M$.
	\end{defi}
	
		We fix a Riemannian metric $g$ on $\Delta$, and we
	introduce the energy functional $J:\Omega\to \R$ by the formula
	\be\label{eq:J}
		J(\gamma)=\frac{1}{2}\int_0^1g(\dot{\gamma}(t),\dot{\gamma}(t))dt.
	\ee
	$J$ is evidently smooth on $\Omega$, but only lower semicontinuous with respect to the weak $W^{1,2}$-topology.
	
	\begin{defi}\label{defi:extee}
		We define the extended endpoint map $\mc{F}:\Omega\to M\times \R$ by
		\be
			 \mc{F}(\gamma)=(F(\gamma),J(\gamma)).
		\ee
	\end{defi}
	Given a point $y\in M$ different from $x_0$, we consider the problem of finding the admissible curves $\gamma$ connecting $x_0$ and $y$ that minimize the energy $J$. This problem can be reformulated as a constrained minimum problem on $\mc{F}$: a curve $\gamma$ is a candidate minimizer if there exists a nonzero covector $\xi=(\lambda,\lambda_0)\in T^*_{F(\gamma)}M\times \R$, defined up to scalar multiples, such that
	\be\label{eq:PMP}
	\xi d_\gamma\mc{F}=\lambda d_{\gamma}F+\lambda_0 d_\gamma J=0.
	\ee

	Candidate minimizing curves are called \emph{extremals}, and the Pontryagin Maximum Principle \cite{Pontryagin} (PMP, in short) characterizes extremals in terms of nowhere zero, absolutely continuous curves $\eta:[0,1]\to T^*M$, called \emph{biextremals} or \emph{extremal lifts}: an admissible curve $\gamma$ is an extremal only if it is the projection on $M$ of a
	biextremal.
	
	Notice that the initial datum $\lambda\in T^*_{F(\gamma)}M$ of every biextremal curve $\eta$ is the first component of a Lagrange multiplier $\xi=(\lambda,\lambda_0)$ in \eqref{eq:PMP}. If $\xi$ is such that $\lambda_0=-1$, we say that $\eta$ is a \emph{normal} biextremal and $\gamma$ is a \emph{normal} extremal curve. In this case small pieces of $\gamma$ are geodesics in the classical sense, i.e. sufficiently short pieces of $\gamma$ are minimizing curves between their endpoints.
	 If instead $\lambda_0=0$, $\gamma$ is a singular curve and it is the projection of an \emph{abnormal} biextremal starting at $\lambda$, as we will now explain.
	
	The cotangent bundle $T^*M$ is canonically endowed with a symplectic form $\omega$, that is a closed non-degenerate smooth section  of $\Lambda^2(T^*M)$, and a bundle projection $\pi:T^*M\to M$.
	Consider the subspace
	\be
		\Delta^\perp:=\{ \lambda\in T^*M\mid \langle\lambda,v\rangle=0,\,\text{for every }v\in\Delta \}\subset T^*M,
	\ee
	where the notation $\langle\cdot,\cdot\rangle$ stands for the duality product between vectors and covectors. The restriction $\ov{\omega}$ of $\omega$ to $\Delta^\perp$ no longer needs to be non-degenerate and may admit characteristic lines \cite{Mont02}.
	\begin{defi}\label{defi:abn}
		A nowhere zero, absolutely continuous curve $\eta:I\to \Delta^\perp$ is an abnormal biextremal if, for a.e. $t\in I$, $\dot{\eta}(t)$ belongs to $\ker(\ov{\omega}_{\eta(t)})$, that is if for a.e. $t\in I$,
		\be
			\ov{\omega}_{\eta(t)}(\dot{\eta}(t),\nu)=0
		\ee
		for every $\nu\in T_{\eta(t)}
		\Delta^\perp$.
	\end{defi}
	
	\begin{prop}[\protect{\cite[Theorem 10]{Hsu92}}]\label{prop:SingAbn}
		An admissible curve $\gamma\in\Omega$ is singular if and only if $\gamma$ is the projection of an abnormal biextremal $\eta:[0,1]\to \Delta^\perp$. As a matter of terminology, we say that $\eta$ is an abnormal lift of $\gamma$.
	\end{prop}
	An admissible curve $\gamma$ may be at the same time both normal and singular,
		and we say that an admissible curve $\gamma$ is strictly singular (or strictly abnormal) if, for every $\xi=(\lambda,\lambda_0)\in (T^*_{F(\gamma)}M\times \R)$ such that \eqref{eq:PMP} holds, $\lambda_0=0$.
	For a further discussion on these points, we refer e.g. to \cite{AgraBook,Pontryagin,Rif14}.

	 Let $\gamma\in\Omega$ be a singular curve. It is well-known \cite{AS96,CJT06} that
	 for every abnormal lift
	$\eta:I\to\Delta^\perp$
	 the Goh condition,
	\be\label{eq:Goh_cond}
		\langle \eta(t),[X,Y](\gamma(t))\rangle=0,
	\ee
	holds for every $t\in I$ and every pair $X,Y$ of local smooth sections of $\Delta$, that is $\eta(t)\in(\Delta_{\gamma(t)}^2)^\perp$ for every $t\in I$.

		\begin{defi}\label{defi:nicesing}
		We say that $\gamma\in\Omega$ is a regular singular curve if $\gamma$ has an abnormal lift $\eta:I\to \Delta^\perp$ satisfying
		\be
			\eta(t)\in (\Delta_{\gamma(t)}^2)^\perp\setminus (\Delta_{\gamma(t)}^3)^\perp
		\ee
		for every $t\in I$.
		\end{defi}

	Regular singular curves are smooth \cite[Theorem 4.4]{AS96}. Moreover, for every regular singular curve $\gamma$, there exists $0<s\leq 1$ such that $\gamma_s:=\gamma\big|_{[0,s]}$ is a strict local minimizer for the $W^{1,2}$-topology on the space of admissible curves joining $x_0$ and $\gamma(s)$. This property depends just on the sub-Riemannian structure $(M,\Delta)$, and not on the metric chosen on it.

	\subsubsection{Adapted coordinates}
	
	\label{sec:local} We briefly explain how to put coordinates on $\Omega$, locally around a regular singular curve $\gamma$.
	We assume as in \eqref{eq:loc_param} that there exist a neighborhood $\mc{O}_{\text{supp}(\gamma)}\subset M$ of $\gamma$ and $X_1,X_2$ smooth vector fields on $M$ such that:
	\begin{itemize}
		\item $\gamma$ is an integral curve of $X_1$ associated with the control $(1, 0)$, satisfying $\dot{\gamma}(t)=X_1(\gamma(t))$ for every $t\in I$.
		\item $
	\Delta_x=\mathrm{span}\{X_1(x),X_2(x)\}$, for every $x\in \mc{O}_{\text{supp}(\gamma)}$.
	\end{itemize}
	Horizontal curves $\gamma'$ contained in $\mc{O}_{\text{supp}(\gamma)}$ are described a.e. on $I$ by the solutions of the differential system
	\be\label{eq:controlsyst}
				\dot{\gamma'}(t)=u_1(t)X_1(\gamma'(t))+u_2(t)X_2(\gamma'(t)),\ \ \gamma'(0)=x_0,
	\ee
	where $u\in \mc{U}_1\subset L^2(I,\R^2)$ and the open set $\mc{U}_1\subset L^2(I,\R^2)$ is a neighborhood of $(1,0)$ that consists of all the pairs $(u_1,u_2)$ such that the solution to \eqref{eq:controlsyst} exists for every $t\in I$.
	
	Additionally, for every $x\in\mc{O}_{\mathrm{supp}(\gamma)}$ we endow $\Delta_x$ with the Riemannian metric $g_x$ that makes $X_1(x)$ and $X_2(x)$ orthonormal. Then we see from \eqref{eq:J} that
	\be
		J(\gamma')=\frac{1}{2}\int_0^1|u_1(t)|^2+|u_2(t)|^2dt.
	\ee
	
	\begin{defi}\label{defi:locchar}
		A local chart on $\mc{U}_1$ is the choice of a neighborhood $\mc{V}_1\subset L^2(I,\R^2)$ of zero and a system of coordinates
		\be
		(u_1,u_2)\mapsto (1+v_1,v_2)
		\ee
		on $\mc{U}_1$ and centered at $(1,0)$.
	\end{defi}
	
	With the choice of a local chart, admissible curves are in one-to-one correspondence with integral curves of
	\be\label{eq:contsysttr}
		\dot{x}(t)=(1+v_1(t))X_1(x(t))+v_2(t)X_2(x(t)),\ \ x(0)=x_0,
	\ee
	for a.e. $t\in I$ and $v=(v_1,v_2)\in \mc{V}_1$.
	
	Let $A:\mc{V}_1\to\Omega$ be the map that associates to the pair $(v_1,v_2)\in\mc{V}_1$ the only solution, up to zero-measure sets, to \eqref{eq:contsysttr}.
Then $A$ is a submersion and, slightly abusing of the notation, we define on $\mc{V}_1$
 $F(v):=F(A(v))$ and $J(v):=J(A(v))$, where $F(v)$ and $J(v)$ are given respectively by:
	\be\label{eq:Jcoord}
		\begin{aligned}
			F(v_1,v_2)&=x_0\circ \eexp\int_0^1(1+v_1(t))X_1+v_2(t)X_2dt,\\
			J(v_1,v_2)&=\frac{1}{2}\int_0^1(1+v_1(t))^2+v_2(t)^2 dt.
		\end{aligned}
	\ee
	
	\begin{defi}
	We say that a control $v\in\mc{V}_1$ is singular if
	\be
		d_{v}F=d_{A(v)}F\circ d_vA
	\ee
	is not surjective.
The corank of $v$ is the codimension of $\IM(d_vF)$ in $T_{F(v)}M$.
	\end{defi}

	\subsection{The endpoint map near regular strictly singular curves}
	
	Let $\gamma\in\Omega(y)$ be a reference regular strictly singular curve,
	and let us choose local coordinates
	centered at $(1,0)$, so that $\gamma$ becomes an integral curve of $\dot{x}(t)=X_1(x(t))$, with $\gamma(0)=x_0$.

By the variation of the constants' formula \cite[Formula (2.28)]{AgraBook} we describe, locally around $\gamma$, the endpoint map $F(v_1,v_2)$ as a perturbation of $y=F(0)$. Setting
	\be\label{eq:gt}
	g_t:=e^{(1-t)X_1}_*X_2,\ \ t\in I,
	\ee
	we write:
	\be\label{eq:endp}
	\begin{aligned}
	F(v_1,v_2)&=x_0\circ \eexp\int_0^1(1+v_1(t))X_1+v_2(t)X_2dt\\&=x_0\circ e^{X_1}\circ\eexp \int_0^1v_1(t)X_1+v_2(t)g_tdt\\&=y\circ\eexp \int_0^1v_1(t)X_1+v_2(t)g_tdt.
	\end{aligned}
	\ee

	\subsubsection{First-order conditions}

	 We see from \eqref{eq:endp} that the differential $d_{0}F$  is given by
         (see \cite[Section 4]{AgraBook}):
         \be\label{eq:diff}
		d_{0}F(v)=\int_0^1v_1(t)dtX_1(y)+\int_0^1v_2(t)g_t(y)dt
	\ee
	for every $v\in L^2(I,\R^2)$.
	
	We split the space of controls $L^2(I,\R^2)$ as the orthogonal sum $L^2(I,\R^2)=\ker (d_0F)\oplus E$,
	where $E$ is a finite-dimensional complement of $\ker (d_0F)$, isomorphic to $\IM(d_0F)$ via the differential $d_0F$.

	Given a linear functional $L:L^2(I,\R^2)\to \R$, its gradient $\nabla_0 L\in L^2(I,\R^2)$ is defined by the equality
	\be
		\left(\nabla_0 L,v\right)_{L^2(I,\R^2)}=d_0 L(v)
	\ee
	for every $v\in L^2(I,\R^2)$, where $(\cdot,\cdot)_{L^2(I,\R^2)}$ is the standard scalar product on $L^2(I,\R^2)$.
	
	\begin{lemma}\label{lemma:what}
		Let $\gamma$ be a
		regular strictly singular curve. Then $\mathcal{J}:=\proj_{\ker (d_0 F)}\nabla_{0}J$ is not zero.
	\end{lemma}

	\begin{proof}
		Consider the orthogonal decomposition
		\be
			L^2(I,\R^2)=\big( C_1\oplus Z_1 \big)\oplus \big( V_2\oplus W_2 \big),
		\ee
	  	where $V_2:=\pi_{\{0\}\oplus L^2(I,\R)}\ker(d_0F)$ is the projection of $\ker(d_0F)$ onto $\{0\}\oplus L^2(I,\R)$, $W_2$ is its complement in $\{0\}\oplus L^2(I,\R)$, i.e. $\{0\}\oplus L^2(I,\R)=V_2\oplus W_2$ and $C_1\oplus Z_1 = L^2(I,\R)\oplus \{0\}$, with $C_1$ the constants and $Z_1$ the controls of mean zero.

		We claim that
		\be\label{eq:claimim}
			\IM (d_0F)=\left\{  d_0F(0,w_2)\mid (0,w_2)\in W_2 \right\},
		\ee
		and we first show how to use this claim to conclude the argument.

		By \eqref{eq:claimim}, there exists $(0,w^0_2)\in W_2$
		such that:
		\be
			d_0F(0,w^0_2)=\int_0^1w^0_2(t)g_t(y)dt=X_1(y)\in\IM (d_0F).
		\ee
		Calling $\widehat{w}:=(1,-w^0_2)\in L^2(I,\R^2)$, we have
		\be\label{eq:kerF}
		\ker (d_0F)=Z_1 \oplus \R\widehat{w} \oplus V_2.
		\ee
		Then, since $\nabla_{0}J=(1,0)$ (compare with \eqref{eq:Jcoord}), we deduce that
		\be
			\mathcal{J}=\proj_{\ker (d_0 F)}\nabla_{0}J=\widehat{w}\neq 0,
		\ee
		and the lemma follows.

		Now we prove the claim, and we reason by contradiction assuming that
		\be
			\mc{S}:=\left\{  d_0F(0,w_2)\mid (0,w_2)\in W_2 \right\}
		\ee
		is a proper subspace of $\IM (d_0F)$. In particular, \eqref{eq:diff} implies that  $X_1(y)$ is the only vector in $\IM (d_0F)$ which is not covered by elements in $\mc{S}$.
		
		Let $\lambda\in T_{F(0)}^*M$ be such that
		\be
		 	\begin{aligned}
				\langle \lambda, d_0F(0,w_2)\rangle&=0 \ \ \text{for every } (0,w_2)\in W_2,\\
				\langle\lambda, X_1(y)\rangle&\neq 0.
			\end{aligned}
		\ee
		Trivially, $\lambda\not\in \IM (d_0F)^\perp$.

		From \eqref{eq:diff} we see that $\{d_0F(v_1,0)\mid v_1\in L^2(I,\R)\}=\mathrm{span}\{d_0F(1,0)\}=\mathrm{span}\{X_1(y)\}$, and that
		\be
			\begin{aligned}
				\langle\lambda, d_0F(v_1,0)\rangle&=v_1\langle \lambda,X_1(y)\rangle\ \ \text{for every } (v_1,0)\in C_1,\\
				\langle\lambda, d_0F(v_1,0)\rangle&=0\ \ \text{for every } (v_1,0)\in Z_1.
			\end{aligned}		
		\ee
		From \eqref{eq:Jcoord}, instead, we have the identity:
		\be
			d_0J(v)=\left(\nabla_{0}J, v\right)_{L^2(I,\R^2)}=\left( 1,v_1\right)_{L^2(I,\R)}=\int_0^1v_1(t)dt.
		\ee
		Choosing $\lambda_0\in \R\setminus\{0\}$ that satisfies:
		\be
			\langle \lambda, X_1(y)\rangle=\langle \lambda, d_0F(1, 0)\rangle=\langle\lambda_0, d_{0}J(1, 0)\rangle=\lambda_0,
		\ee
		we see that $\xi=(\lambda,-\lambda_0)\in T_{F(0)}^*M\times \R$ is a normal covector, in the sense that
		\be
			\xi d_{0}\mc{F}(v)=\lambda d_{0}F(v)-\lambda_0 d_{0}J(v)
		\ee
		for every $v\in L^2(I,\R^2)$. The absurd follows since we assumed $\gamma$ strictly singular.
	\end{proof}
	
	\begin{cor}\label{cor:dec1}
		Let $\gamma$ be a
		regular strictly singular curve. Then the following orthogonal decomposition holds:
		\be\label{eq:splitker}
			L^2(I,\R^2)=\ker(d_0F)\oplus E=Z_1\oplus \R\mathcal{J}\oplus V_2\oplus E,
		\ee
		where:
		\begin{itemize}
			\item  $Z_1\subset L^2(I,\R)\oplus \{0\}$ is the set of controls with zero mean,
			\item $\mathcal{J}=\proj_{\ker(d_0F)}\nabla_0J$, and
			\item  $V_2=\pi_{\{0\}\oplus L^2(I,\R)}\ker(d_0F)$.
		\end{itemize}

		In particular, for the extended endpoint map $\mc{F}=(F,J)$ we have:
		\be
			\begin{aligned}
				\ker(d_0\mc{F})=\ker (d_{0}F)\cap \ker (d_{0}J)=Z_1\oplus V_2,
			\end{aligned}
		\ee
		and $\ker(d_0\mc{F})$ is of codimension one in $\ker (d_{0}F)$.
	\end{cor}
	
	\subsubsection{Second-order conditions} We introduce in this section the Hessian map
	\be
		\He_0F:\ker(d_0F)\to T_{F(0)}M/ \IM(d_0F)
	\ee
	or, to be more precise, its scalar projections $\{\lambda\He_0F\mid \lambda\in\IM(d_0F)^\perp\}$.

	Given $\lambda\in\IM(d_0F)^\perp$, we have
	\be
		\begin{aligned}
		\lambda\He_0F:\ker(d_0F)&\to \R,\\
					v&\mapsto \left\langle\lambda,\left(\int_0^1\int_0^t(v_1(\tau)X_1+v_2(\tau)g_\tau)\circ (v_1(t)X_1+v_2(t)g_t)d\tau dt\right)(y)\right\rangle.
		\end{aligned}
	\ee

	Moreover, \eqref{eq:diff} plus the fact that
	$\lambda\in\IM (d_0F)^\perp$ imply that
	\be\label{eq:condperp}
		\left\{\begin{aligned}
	\langle\lambda,X_1(y)\rangle &=0,\\
	\langle\lambda,g_t(y)\rangle&\equiv 0 \ \ \text{for every } t\in I,
		\end{aligned}\right.
	\ee
	and then we deduce by \eqref{eq:gt} that
	\be\label{eq:condperp2} 			
	\frac{d}{dt}\langle\lambda,g_t(y)\rangle=-\langle\lambda,[X_1,g_t](y)\rangle\equiv 0
	\ee
	for every $t\in I$, i.e. $\langle\lambda,[X_1,g_t](y)\rangle\equiv 0$ for every $t\in I$.

	Combining \eqref{eq:condperp} and \eqref{eq:condperp2}, the Hessian takes the form \cite[Exercise 20.4]{AgraBook}:
	\be\label{eq:Q}\begin{aligned}
	\lambda\He_0F(v)&=\int_{0}^{1}\left\langle\lambda,\left[\int_0^t v_1(\tau)X_1+v_2(\tau)g_\tau d\tau,v_1(t)X_1+v_2(t)g_t\right](y)\right\rangle dt\\&=\int_{0}^{1}\left\langle\lambda,\left[\int_0^t v_2(\tau)g_\tau d\tau,v_2(t)g_t\right](y)\right\rangle dt
	\end{aligned}\ee
	for every $v\in \ker(d_0F)$.
	
	Notice that the Hessian leads to a second orthogonal decomposition of the control space (compare with \eqref{eq:splitker}):
	\be\label{eq:splitq}
		L^2(I,\R^2)=\ker (d_0F)\oplus E=P(\lambda)\oplus N(\lambda)\oplus Z(\lambda)\oplus E,
	\ee
	where
	\begin{itemize}
		\item  $P(\lambda)$ is the positive eigenspace of $\lambda \He_0F$,
		\item  $N(\lambda)$ is the negative eigenspace of $\lambda \He_0F$,
		\item  $Z(\lambda):=\ker(\lambda \He_0F)$.
	\end{itemize}
	
	\begin{remark}\label{rem:properties_decomposition}
	By \cite[Theorem 7.1]{HestQ}, we assume this decomposition to be also orthogonal with respect to $\lambda\He_0 F$. Notice moreover that $\lambda\He_0 F$ does not depend on $v_1$: this implies that, for every $\lambda\in\IM(d_0F)^\perp$ the subspace $Z_1\subset \ker(d_0F)$, consisting of the controls with zero mean, is contained in $Z(\lambda)$.
	\end{remark}
	
	\begin{remark}
		Since $\gamma$ is strictly singular, we have that $\IM(d_0\mc{F})^\perp=\{(\lambda,0)\mid \lambda\in \IM(d_0F)^\perp\}\subset T^*_{F(0)}M\times \R$. This implies that, given $\lambda\in \IM(d_0F)^\perp$, \eqref{eq:Q} holds also for $\lambda \He_0\mc{F}:\ker(d_0\mc{F})\to \R$,
		the only difference being the domain $\ker(d_0\mc{F})$ strictly smaller than $\ker(d_0F)$.
	\end{remark}
	
	\section{Conjugate points along corank-one regular strictly singular curves}\label{sec:conjpoints}

	\subsection{Completion space and conjugate points}
	Fix a corank-one regular strictly singular curve $\gamma(t):=x_0\circ e^{tX_1}$, $t\in I$. Then $\IM(d_{0}F)^\perp$ has dimension one in $T^*_{F(0)}M$, and is generated by a norm-one covector $\lambda$. In particular, the subspace of covectors $\xi\in T^*_{F(0)}M\times \R$ that satisfy \eqref{eq:PMP} coincides with $\mathrm{span}\{(\lambda,0)\}$, and thus
	$\gamma$ admits a unique extremal lift $\eta$ up to real multiples, which is necessarily abnormal.

	We return to the expression of $F(v)$, this time interpreting it as a perturbation of the flow
	\be
		x_0\circ \eexp \int_0^1 v_2(t) X_2dt=x_0\circ e^{\int_0^1 v_2(t)dt X_2}.
	\ee
	
	By the variation of the constants' formula \cite[Formula (2.27)]{AgraBook}, setting $w_2(t):=\int_0^t v_2(\tau)d\tau$ for every $t\in I$, we have:
	\be\label{eq:changeofcoord}\begin{aligned}
		F(v)&=x_0\circ \eexp\int_0^1(1+v_1(t))X_1+v_2(t)X_2dt\\ &=x_0\circ \eexp\int_0^1(1+v_1(t))e_*^{-w_2(t)X_2}X_1dt\circ e^{w_2(1)X_2}.
	\end{aligned}\ee
	In particular, $F=F(v_1,w_2)$ can be thought as a map on the space $L^2(I,\R)\oplus H^1(I,\R)$. One last application of \cite[Formula (2.28)]{AgraBook} leads to
	\be\label{eq:Gnewcoord}
		\begin{aligned}
			F(v_1,w_2)&=x_0\circ \eexp\int_0^1(1+v_1(t))e_*^{-w_2(t)X_2}X_1-X_1+X_1dt\circ e^{w_2(1)X_2}\\
					 &=y\circ \eexp\int_0^1(1+v_1(t))e_*^{-w_2(t)g_t}X_1-X_1dt\circ e^{w_2(1)X_2},
		\end{aligned}
	\ee
	where now the flow $\eexp\int_0^1(1+v_1(t))e_*^{-w_2(t)g_t}X_1-X_1dt$ is seen as a perturbation of
	\be
		x_0\circ \eexp\int_0^1 X_1dt\circ e^{w_2(1)X_2}=x_0\circ e^{X_1}\circ e^{w_2(1)X_2}=y\circ e^{w_2(1)X_2}.
	\ee

	The first order term in the expansion of $F$, i.e. $d_0F(v_1,w_2)$ is given by:
	\be\label{eq:kerG}
		\begin{aligned}
		d_0F(v_1,w_2)&=\int_0^1 v_1(t)dt X_1(y)+w_2(1)X_2(y)-\int_0^1 w_2(t)[g_t,X_1](y)dt\\
				       &=\int_0^1 v_1(t)dt X_1(y)+w_2(1)X_2(y)-\int_0^1 w_2(t)\dot{g}_t(y)dt,
		\end{aligned}
	\ee
	implying the existence of a constant $C>0$ such that, for every $(v_1,w_2)\in \ker (d_0F)$, one has:
	\be\label{eq:kerest}
		\left| \int_0^1 v_1(t)dt \right|\leq C\|w_2\|_{L^2(I,\R)},\ \ |w_2(1)|\leq C\|w_2\|_{L^2(I,\R)}.	
	\ee
	
	Integrating by parts in \eqref{eq:Q}, we obtain the expression of $\lambda \He_0F$ on $L^2(I,\R)\oplus H^1(I,\R)$:
	\be\label{eq:qfs}
		\begin{aligned}
		\lambda \He_0F(v_1,w_2)&=\int_0^1\left\langle \lambda, \left[w_2(t)g_t-\int_0^tw_2(\tau)\dot{g}_\tau d\tau, v_2(t)g_t\right](y)\right\rangle dt\\
		&=-\int_0^1\left\langle \lambda,\left[w_2(t)\dot{g}_t,\int_t^1v_2(\tau)g_\tau d\tau\right](y)\right\rangle dt\\
		&=-\int_0^1\left\langle\lambda, \left[w_2(t)\dot{g}_t,w_2(1)X_2-w_2(t)g_t-\int_t^1 w_2(\tau)\dot{g}_\tau d
		\tau\right](y)\right\rangle dt\\
		&=\int_0^1\langle \lambda,[\dot{g}_t,g_t](y)\rangle w_2(t)^2dt+\int_0^1\left\langle 				\lambda,\left[w_2(1)X_2+\int_0^tw_2(\tau)\dot{g}_\tau d\tau,w_2(t)\dot g_t\right](y)\right\rangle dt.
	\end{aligned}
	\ee
	
	We claim that, up to considering the covector $-\lambda$ instead of $\lambda$, there exists a constant $\kappa>0$ such that the Legendre condition
	\be\label{eq:kappa}
		\langle \lambda,[\dot{g}_t,g_t](y)\rangle\geq \kappa
	\ee
	holds for every $t\in I$.
	
	In fact, assume \eqref{eq:kappa} to be zero for some $t\in I$. Recalling that for the abnormal lift $\eta:I\to \Delta^\perp$ of $\gamma$ there holds the identity
	\be\label{eq:eta_s}
		\eta(s)=\left(e^{(1-s)X_1}\right)^*\lambda\in \Delta_{\gamma(s)}^\perp\subset T^*_{\gamma(s)}M,
	\ee
	for every $s\in I$, we deduce:
	\be
		\begin{aligned}
		0&=\langle \lambda,[g_t,\dot{g}_t](y)\rangle\\
		  &=\langle \lambda,[e^{(1-t)X_1}_* [X_1,X_2],e^{(1-t)X_1}_* X_2](y)\rangle\\
		  &=\langle \eta(t), [[X_1,X_2],X_2](\gamma(t))\rangle.		
		\end{aligned}
	\ee
	Combining this with \eqref{eq:condperp2}, we conclude that $\eta(t)\in \Delta_{\gamma(t)}^3$, contradicting Definition~\ref{defi:nicesing} and the fact that $\gamma$ is regular.	
	
	We endow $H^1(I,\R)$ with the norm $\|\cdot\|_{2}$ defined by:
	\be\label{eq:norm}
	\left \|w\right\|_{2}=\left| w(1) \right|_{\R}+\|w\|_{L^2(I,\R)},
	\ee
	and we recall that the completion of $H^1(I,\R)$ with respect to $\|\cdot\|_2$ is isomorphic $\R\oplus L^2(I,\R)$.	
	
	By \eqref{eq:Gnewcoord}, $F$ admits a continuous extension $\mathring{F}$ to $L^2(I,\R)\oplus \R\oplus L^2(I,\R)$:
	\be
		\mathring{F}(v,c,w):=y\circ \eexp\int_0^1(1+v(t))e_*^{-w(t)g_t}X_1-X_1dt\circ e^{c X_2},
	\ee
	for every $(v,c,w)\in \mathring{\mc{V}}_1$, where $\mathring{\mc{V}}_1\subset L^2(I,\R)\oplus \R\oplus L^2(I,\R)$ is a suitable local chart around the origin, in the sense of Definition~\ref{defi:locchar}.
	
	Formulas \eqref{eq:kerG} and \eqref{eq:qfs}
	immediately extend to $d_0\mathring{F}$ and $\lambda\He_0\mathring{F}$. For every $(v,c,w)\in\ker(d_0\mathring{F})$ the following holds:
		\be\label{eq:kerestc}
			\left| \int_0^1 v(t)dt \right|\leq C\|w\|_{L^2(I,\R)},\ \ |c|\leq C\|w\|_{L^2(I,\R)}.
		\ee
		
		According to formula \eqref{eq:splitq}, we call $\mathring{P}$ (resp. $\mathring{N}$) the positive (resp. the negative) eigenspace of $\lambda\He_0\mathring{F}$, and $\mathring{Z}:=\ker(\lambda\He_0\mathring{F})$ its kernel.
		
		We deduce from \eqref{eq:qfs} that $\lambda\He_0\mathring{F}$ does not depend on the $v$-coordinate, implying that the subspace
		\be
			Z_1:=\left\{(v,0,0)\mid
			v\in L^2(I,\R),\, \int_0^1v(t)dt=0\right\}\subset \ker(d_0\mathring{F})
		\ee
		is contained in $\mathring{Z}$.
		Observe that this notation is consistent with that of Corollary~\ref{cor:dec1}.

	\begin{defi}[Conjugate points]\label{defi:conj}
		Let $\gamma$ be a corank-one regular strictly singular curve. We say that $y=\gamma(1)$ is a conjugate point along
		$\gamma$ if
	$Z_1$ is a proper subspace of $\mathring{Z}$. We call $k:=\dim(Z_1')$ the multiplicity of $y$ as a conjugate point, where $Z_1'$ is the $(L^2(I,\R)\oplus \R\oplus L^2(I,\R))$-orthogonal complement of $Z_1$ in $\mathring{Z}$, i.e. $Z_1'$ satisfies:
	\be
		\mathring{Z}=Z_1\oplus Z_1'.
	\ee
	\end{defi}
	
	The following proposition collects some useful results from \cite[Theorem 1]{Sarysec}.
	
	\begin{prop}\label{prop:ConjProp}
		Let $\gamma$ be a corank-one regular strictly singular curve. Defining, for $s\in I$, $\gamma_s:=\gamma\big|_{[0,s]}$ and $\eta(s)$ as in \eqref{eq:eta_s}, we have:
		\begin{itemize}
			\item [(i)] There exists $s_0\in I$ such that the Hessian map $\eta(s)\He_{\gamma_s} \mathring{F}$ is positive definite for every $s\le s_0$.
			\item [(ii)] Conjugate points along $\gamma$ are isolated, and every conjugate point has a finite multiplicity.
			\item[(iii)] The negative index of $\lambda\He_0\mathring{F}$ equals the sum of the multiplicities of all conjugate points along $\gamma$. In particular, it is finite.
		\end{itemize}
	\end{prop}
	
	From now on, we assume $y$ to be not conjugate along $\gamma$, i.e. following Definition~\ref{defi:ranktwonice}, we
	suppose $\gamma$ to be rank-two-nice, and thus
	\be\label{eq:eq_kers}
		\mathring{Z}=\ker(\lambda \He_0\mathring{F})=\left\{(v,0,0)\mid
		v\in L^2(I,\R),\, \int_0^1v(t)dt=0\right\}=Z_1.
	\ee
 	
	\subsection{Analytical properties of the Hessian map}\label{sec:an_prop}
	We study in this section how $\lambda \He_0\mathring{F}$ behaves on the subspace $\mathring{P}\oplus \mathring{N}\subset L^2(I,\R)\oplus\R\oplus L^2(I,\R)$, equipped with the product norm and a scalar product $(\cdot,\cdot)$ induced by the norm.		
	\begin{prop}\label{prop:operators}
		Let $T$ be the linear operator on $\mathring{P}\oplus \mathring{N}$ defined by the equality
	\be
			\left(T(v,c,w),(v,c,w)\right):=\int_0^1\left\langle \lambda,\left[cX_2+\int_0^tw(\tau)\dot{g}_\tau d\tau,w(t)\dot g_t\right](y)\right\rangle dt.
	\ee
	for every $(v,c,w)\in \mathring{P}\oplus \mathring{N}$. Then $T$ is a compact and self-adjoint operator.
\end{prop}
	
	\begin{proof}
		To polarize $T$, we recall the identity
		\be
			\int_0^1\left\langle \lambda,\left[\int_0^t a(\tau)X_\tau d\tau,b(t)Y_t \right](y)\right\rangle dt=\int_0^1\left\langle \lambda,\left[ a(t)X_t,\int_t^1b(\tau)Y_\tau d\tau \right](y)\right\rangle dt,
		\ee
		which is valid for every $t\in I$, $a,b\in L^2(I,\R)$ and vector fields $X_t, Y_t\in\Vc(M)$.

		We deduce that $T=\proj_{\mathring{P}\oplus \mathring{N}}T'$, where:
	\be\label{eq:decomp_operators}\begin{aligned}
			T':\mathring{P}\oplus\mathring{N}&\to L^2(I,\R)\oplus \R\oplus L^2(I,\R),\\
			(v,c,w)&\mapsto \begin{pmatrix} 0 \\ \int_0^1 \left\langle \lambda,[X_2,\dot g_t](y)\right\rangle w(t)dt \\ c\langle \lambda,[X_2,\dot g_t](y)\rangle+\left\langle \lambda,\left[\int_0^t w(\tau)\dot{g}_\tau d\tau,\dot g_t \right](y) \right\rangle-\left\langle \lambda,\left[\int_t^1w(\tau)\dot g_\tau d\tau,\dot{g}_t\right](y)  \right\rangle\end{pmatrix}
	\end{aligned}\ee
	for every $(v,c,w)\in \mathring{P}\oplus \mathring{N}$, and $\proj_{\mathring{P}\oplus \mathring{N}}$ stands for the projection of $L^2(I,\R)\oplus \R\oplus L^2(I,\R)$ onto $\mathring{P}\oplus \mathring{N}$.

	The self-adjointness of $T$ is clear, after all it is a linear operator associated with a bilinear form.
	
	Let
	\be
		K(t,\tau):=\langle \lambda,[\dot{g}_\tau,\dot g_t](y)\rangle\chi_{[0,t]}(\tau)-\langle \lambda,[\dot g_\tau,\dot{g}_t](y)\rangle\chi_{[t,1]}(\tau),\ \ K(t,\tau)\in L^2(I^2,\R).
	\ee
	To prove the compactness of $T$, it suffices to show that
	the mapping in the last component of $T'$,
		\be\begin{aligned}
			T'_3:\mathring{P}\oplus\mathring{N}&\to L^2(I,\R),\\
			(v,c,w)&\mapsto c\langle \lambda,[X_2,g_t](y)\rangle+\int_0^1K(t,\tau)w(\tau)d\tau,
			\end{aligned}
		\ee
		is compact. In fact, $(v,c,w)\mapsto c\langle \lambda,[X_2,\dot g_t](y)\rangle$ is a rank-one operator, while the compactness of
		\be
			(v,c,w)\mapsto \int_0^1K(t,\tau)w(\tau)d\tau
		\ee
		is classical, and proved e.g. in \cite[Chapter 6]{HirschLacombe}. The thesis follows.
	\end{proof}

	For our next result we define $(\lambda \He_0\mathring{F})_{\mathring{P}}$ to be the restriction of $\lambda \He_0\mathring{F}$ to its positive eigenspace $\mathring{P}$, and recall that if $(v,c,w)\in\mathring{P}$, then $v\in L^2(I,\R)$ is constant by \eqref{eq:eq_kers}.
	
	\begin{prop}\label{prop:positive}
		There exists a constant $K>0$ such that,
		\be
			(\lambda \He_0\mathring{F})_{\mathring{P}}(v,c,w)\ge K\|(v,c,w)\|^2
		\ee
		for every $(v,c,w)\in\mathring{P}$.
	\end{prop}	
	
	\begin{proof}
		Let
		$T_{\mathring{P}}$ denote the restriction of
		$T$ to $\mathring{P}$. Define
		\be\begin{aligned}
		\alpha:&=\inf\left\{ (\lambda \He_0\mathring{F})_{\mathring{P}}(v,c,w)\mid (v,c,w)\in\mathring{P},\ \ \|(v,c,w)\|^2=1 \right\}\\
			  &=
		1+\inf\left\{ \left(  T_{\mathring{P}}(v,c,w),(v,c,w)\right)\mid (v,c,w)\in\mathring{P},\ \ \|(v,c,w)\|^2=1 \right\}.
		\end{aligned}\ee
		
		Clearly $\alpha\geq 0$. We claim that, in fact, $\alpha>0$. If this is true, by \eqref{eq:kappa}, \eqref{eq:kerestc} and the fact that $v$ is constant, denoting $\ell:=\max_{t\in I}\left\langle \lambda,[\dot{g}_t,g_t](y)   					\right\rangle>0$ we have:
		\be
			\begin{aligned}
			(\lambda \He_0\mathring{F})_{\mathring{P}}(v,c,w)&\ge \frac{\alpha}{\ell}\int_0^1\left\langle \lambda,[\dot{g}_t,g_t](y)   					\right\rangle w(t)^2dt
			\\ &\ge \frac{\kappa\alpha}{\ell}\|w\|^2_{L^2(I,\R)}\ge \frac{\kappa\alpha}{\ell}\left( \frac{\|w\|^2_{L^2(I,\R)}}{3} +\frac{|c|^2}{3C^2}+\frac{\left|\int_0^1 v(t)dt\right|^2}{3C^2}\right)
			\\ &\ge\frac{\kappa\alpha}{\ell} \min\left\{\frac{1}{3},\frac{1}{3C^2}\right\}\|(v,c,w)\|^2,
			\end{aligned}
		\ee
		and we conclude with $K:=\frac{\kappa\alpha}{\ell} \min\left\{\frac{1}{3},\frac{1}{3C^2}\right\}$.
		
		Now we prove the claim, assuming by contradiction that $\alpha=0$. Notice that $T_{\mathring{P}}$ is
compact and self-adjoint by Proposition~\ref{prop:operators}.
 Its eigenvalues are bounded, countable, and can only accumulate at zero (see e.g., \cite{Kato}).

 Clearly,
		\be
			-1=\inf\left\{ \left(  T_{\mathring{P}}(v,c,w),(v,c,w)\right)\mid (v,c,w)\in\mathring{P},\ \ \|(v,c,w)\|^2=1 \right\},
		\ee
		and therefore $-1$ coincides with the lowest bound of the spectrum $\sigma(T_{\mathring{P}})$, which is actually an eigenvalue by the Fredholm alternative. Then, we deduce the existence of $(v,c,w)\in\mathring{P}$ such that $\|(v,c,w)\|^2=1$ and such that
\be
	(\lambda \He_0\mathring{F})_{\mathring{P}}(v,c,w)=0.
\ee

Since $(\lambda \He_0\mathring{F})_{\mathring{P}}$ is positive on $\mathring{P}$, by \cite[Lemma 6.2]{HestQ} the linear application\footnote{By a slight abuse of the notation, we identify $(\lambda \He_0\mathring{F})_{\mathring{P}}$ and its associated bilinear map.}
\be
	(\lambda\He_0F)_{\mathring{P}}((v,c,w),\cdot ):\mathring{P}\to \R
\ee
is the zero map. Since $\mathring{P}$, $\mathring{N}$ and $\mathring{Z}$ are orthogonal with respect to $\lambda\He_0F$ (compare with Remark~\ref{rem:properties_decomposition}), we conclude that $(v,c,w)\in \mathring{Z}$, whence the absurd. 
	\end{proof}

	Proposition~\ref{prop:positive} implies that the eigenvalues of $(\lambda \He_0\mathring{F})_{\mathring{P}}$ do not accumulate towards zero on $\mathring{P}$. Since this is trivially true on the finite dimensional negative eigenspace $\mathring{N}$ (see (iii) of Proposition~\ref{prop:ConjProp}), we deduce the following result.
	
	\begin{prop}\label{prop:invertible}
		Let $(\lambda \He_0\mathring{F})_{\mathring{P}\oplus \mathring{N}}$ denote the restriction of $\lambda \He_0\mathring{F}$ to $\mathring{P}\oplus \mathring{N}$, and let $L:\mathring{P}\oplus \mathring{N}\to\mathring{P}\oplus \mathring{N}$ be the linear operator satisfying
		\be
			\left\langle L(v,c,w),(v,c,w)\right\rangle=(\lambda \He_0\mathring{F})_{\mathring{P}\oplus \mathring{N}}(c,w)
		\ee
		for every $(v,c,w)\in\mathring{P}\oplus \mathring{N}$. Then $L$ is bounded from below, hence invertible, on $\mathring{P}\oplus \mathring{N}$.
	\end{prop}

	\section{A preliminary change of coordinates}\label{sec:deformation}
		
	We turn in this section to the Banach space $\BS$.
		
	\subsection{Cutting the kernel of the Hessian map} Let us define the subspace
	\be\label{eq:Banspacezero}
		L^\infty_Z(I,\R):=\left\{ v \in L^\infty(I,\R)\mid \int_0^1v(t)dt=0  \right\}\subset L^\infty(I,\R).
	\ee
	
	Every element $v\in L^\infty(I,\R)$ decomposes as $v=v^C+v^Z$, where
	\begin{itemize}
		\item  $v^Z\in L^\infty_Z(I,\R)$,\vspace{0.1cm}
		\item $v^C=\int_0^1v(t)dt\in \R$,
	\end{itemize}
	and $v^C$ and $v^Z$ are orthogonal with respect to the $L^2(I,\R)$-product.
	Accordingly, we fix coordinates $(v_1^C,v_1^Z,v_2)$ on $\BS$ throughout this section. 
	
	\noindent Given $\alpha>0$, we consider the open sets $\mc{V}_2,\mc{V}_3\subset \BS$ given by:
	\be\label{eq:vsets}
		\begin{aligned}
			\mc{V}_2&=\left\{v\in \BS\mid 1+v_1^Z> \alpha,\ \ \text{a.e. on } I\right\},\\
			\mc{V}_3&=\left\{w\in \BS\mid 1+\frac{w_1^Z}{1+w^C_1}> \alpha,\ \ \text{a.e. on } I\right\}.
	\end{aligned}
	\ee
	
	\begin{defi} Let us set, for every $v\in \BS$ and every $t\in I$,
	\be
			\phi_v(t):=\int_0^t1+v^Z_1(\tau)d\tau.
		\ee
		
		Then we define the origin-preserving map
		\be\label{eq:rhopar}
			\begin{aligned}
				\rho:\mc{V}_2&\to \mc{V}_3\\
				\left(v^C_1,v^Z_1,v_2\right)&\mapsto \left( v^C_1,(1+v^C_1)v_1^Z, \dot{\phi}_v\left( v_2\circ \phi_v \right)   \right).
			\end{aligned}
		\ee
	\end{defi}
	
	Observe that for every $v\in\mc{V}_2$, the time-reparametrization $\phi_v:I\to I$ is well-defined on the set $S_v:=\{ s\in I\mid s=\phi_v(t)\ \ \text{and}\ \  \dot\phi_v(t)\ \ \text{exists different from zero} \}$,
	which is of full measure by the Sard lemma for real-valued absolutely continuous functions (see, e.g. \cite[Theorem 16]{Varberg}).
	
	Moreover, it is not difficult to compute the inverse map
	\be\label{eq:coordinv}
		\begin{aligned}
			\rho^{-1}:\rho(\mc{V}_2)&\to \mc{V}_2\\
				(w_1^C,w_1^Z,w_2)&\mapsto\left( w_1^C, \frac{w_1^Z}{1+w_1^C},\frac{w_2\circ \phi_w^{-1}}{\dot{\phi}_w\circ \phi_w^{-1}}  \right),
		\end{aligned}
	\ee
	where this time we have
	\be
		\begin{aligned}
			\phi_w:I&\to I\\
				   \phi_w(t)&=\int_0^t1+\frac{w_!^Z(\tau)}{1+w_1^C}d\tau,
		\end{aligned}
	\ee
	and for every $w\in\rho(\mc{V}_2)$, the time-reparametrization $\phi_w:I\to I$ is well-defined on the full-measured set $S_w=\{ s\in I\mid s=\phi_w(t)\ \ \text{and}\ \  \dot\phi_w(t)\ \ \text{exists different from zero}\}$.

	Let us finally remark that, if $\mc{V}_2$ is chosen to be contained in the local chart $\mc{V}_1$ of Definition~\ref{defi:locchar}, then:
	\be\label{eq:changeofvv}
		\begin{aligned}
		F(\rho(v_1,v_2))&=x_0\circ \eexp\int_0^1 (1+v_1^Z(s))(1+v_1^C)X_1+\dot{\phi}_v(s)v_2(\phi_v(s))X_2ds\\&= x_0\circ\eexp\int_0^1 (1+v^C_1)X_1+v_2(t)X_2dt\\ &=y\circ \eexp\int_0^1 v^C_1X_1+v_2(t)g_tdt.
		\end{aligned}
	\ee	
	where the passage from the first to the second line follows by the change of variable $t=\phi_v(s)$, noticing that $1+v_1^Z(s)=\dot{\phi}_v(s)$.
	
	In particular, $F\circ\rho$ has no explicit dependence on the zero-mean part of the control $v_1^Z$.
	
	\subsection{Regularity properties of $\rho$} We prove in this section that $\rho$ is an homeomorphism onto its image.
	
	\begin{lemma}\label{lemma:cont}
		Let $w\in\rho(\mc{V}_2)$, and $(w_n)_{n\in \N}$ be a sequence converging to $w$ in $\BS$.
Then:
		\begin{itemize}
			\item [(i)] $\phi_{w_n}\to \phi_w$ uniformly on $I$;
			\item [(ii)] $\phi_{w_n}^{-1}\to \phi_w^{-1}$ pointwise on $S_w$.
		\end{itemize}
	\end{lemma}
	
	\begin{proof}
 The proof of (i) is trivial and left as an exercise.
		We pass to (ii). Observe that for every $s\in I$ we have $\phi_w^{-1}(s):=\inf\{  t\in I\mid \phi_w(t)=s$ and similarly for $\phi_{w_n}^{-1}(s)$, for $n\in \N$.

		Let $s\in S_w$, $t:=\phi_w^{-1}(s)$, $t_n:=\phi_{w_n}^{-1}(s)$ and assume that $\lim_{n\to\infty}t_n=\ov{t}> t$. By the triangular inequality, for every $\varepsilon>0$ there exists $n_\varepsilon\in\N$ such that for every $n\ge n_\varepsilon$ one has:
		\be\label{eq:abs}
			s=\phi_w(t)\le \phi_w(\ov{t})\le \phi_{w_n}(t_n)+3\varepsilon= s+3\varepsilon.
		\ee
		Notice that this implies $\phi_w(\ov{t})=s$, yielding that $s\in I\setminus S_w$ i.e. either $\phi_w(t)$ does not exists, or it exists and it is equal to zero, as $\phi_w$ would be constant on $[t,\ov{t}]$. Hence the absurd.

		By similar arguments we rule out also the case $\lim_{n\to\infty}t_n=\ov{t}<t$. Indeed we find that for every $\varepsilon>0$, there exists $n_\varepsilon\in\N$ such that for every $n\ge n_\varepsilon$
		\be
			s=\phi_w(t)\ge \phi_w(\ov{t})\geq \phi_{w_n}(t_n)-3\varepsilon=s-3\varepsilon,
		\ee
		and we conclude.
\end{proof}

	\begin{prop}\label{prop:contrep}
		The map $\rho:\mc{V}_2\to\rho(\mc{V}_2)$ is an homeomorphism.
	\end{prop}	
	
	\begin{proof}
	We only prove the continuity of $\rho^{-1}$, since the continuity of $\rho$ follows by similar arguments.
	
	Let $w\in \rho(\mc{V}_2)$ and let $(w_n)_{n\in\N}$ be a sequence converging to $w$ in $\BS$.
	Define $z_n:=\dot{\phi}_{w_n}=1+\frac{w_{n,1}^Z}{1+w_{n,1}^C}$ for every $n\in \N$ and, similarly, $z:=\dot{\phi}_{w}=1+\frac{w_{1}^Z}{1+w_{1}^C}$. Notice that, by construction, $z_n> \alpha$ and $z> \alpha$. Moreover, $(z_n)_{n\in\N}$ converges to $z$ in $L^\infty(I,\R)$.
	
	It is then sufficient to establish that:
		\be\label{eq:mainest}
			\lim_{n\to\infty}	\int_\Sigma\left|  \frac{w_{n,2}(\phi_{w_n}^{-1}(s))}{z_n(\phi_{w_n}^{-1}(s))}  -   \frac{w_2(\phi_w^{-1}(s))}{z(\phi_w^{-1}(s))} \right|^2ds=0,
		\ee
		where $\Sigma\subset I$ is the full-measured set:
		\be
			\Sigma:=\left\{ s\in I\mid s=\phi_w(t)\ \ \text{and}\ \  \exists\dot\phi_w(t)\neq 0 \right\}\cap \bigcap_{n\in\N}\left\{ s\in I\mid s=\phi_{w_n}(t)\ \ \text{and}\ \  \exists\dot\phi_{w_n}(t)\neq 0 \right\}.
		\ee
		
		By the triangular inequality, \eqref{eq:mainest} can be bounded in two steps. We begin with:
		\be\label{eq:xxx}
			\begin{aligned}
				\int_\Sigma\left|  \frac{w_{n,2}(\phi_{w_n}^{-1}(s))}{z_n(\phi_{w_n}^{-1}(s))}  -   \frac{w_2(\phi_w^{-1}(s))}{z(\phi_w^{-1}(s))} \right|^2ds\le 2\bigg(&\int_\Sigma\left|\frac{w_{n,2}(\phi_{w_n}^{-1}(s))}{z_n(\phi_{w_n}^{-1}(s))}-\frac{w_2(\phi_{w_n}^{-1}(s))}{z(\phi_{w_n}^{-1}(s))}\right|^2 ds+\\
				&\int_{\Sigma}\left| \frac{w_2(\phi_{w_n}^{-1}(s))}{z(\phi_{w_n}^{-1}(s))}-\frac{w_2(\phi_w^{-1}(s))}{z(\phi_w^{-1}(s))} \right|^2ds\bigg).
			\end{aligned}
		\ee
		
		By the change of variables $t:=\phi_{w_n}^{-1}(s)$, we bound the first term by
		\be
			\begin{aligned}
			&\int_{\phi_{w_n}^{-1}(\Sigma)}\frac{| z(t)w_{n,2}(t)- z_n(t)w_2(t) |^2}{z_n(t)z(t)^2}dt\le \frac{2}{\alpha^3}	\int_0^1 z(t)^2|w_{n,2}(t)-w_2(t)|^2+|z(t)-z_n(t)|^2|w_2(t)^2dt,
			\end{aligned}
		\ee
		and we conclude that it converges to zero as $n\to \infty$. The convergence to zero of the second term in \eqref{eq:xxx}
		is instead a consequence of Lemma~\ref{lemma:cont} and the Lebesgue's dominated convergence theorem. The result follows.
	\end{proof}

	\section{Normal forms around rank-two-nice singular curves}\label{se:norfor}
	
	We study in this section the endpoint map $F$ on the set $\rho(\mc{V}_2)\subset \BS$, which is an open neighborhood of the origin by Proposition~\ref{prop:contrep}.
	
	By \eqref{eq:changeofvv} we see that, on $\rho(\mc{V}_2)$, $F$ depends only on $\int_0^1v_1(t)dt\in \R$ and $v_2$.
	To find a normal form it is then sufficient to restrict $F$ onto $\R\oplus H^1(I,\R)\subset L^\infty(I,\R)\oplus H^1(I,\R)$, possibly after an integration by parts as in \eqref{eq:changeofcoord}.
	
	\subsection{A generalized Morse Lemma}
	Let $\mathtt{F}:=F\big|_{\R\oplus H_1(I,\R)}$. Then we have:
	\be\label{eq:endpdef}
		\begin{aligned}
		\mathtt{F}(c_1,w_2):&=x_0\circ \eexp\int_0^1(1+c_1)X_1+\dot{w}_2(t)X_2dt\\&=y\circ\eexp \int_0^1c_1X_1+\dot{w}_2(t)g_tdt
		\end{aligned}
	\ee
	on $\R\oplus H^1(I,\R)$, and
	\be\label{eq:endpdefpr}
		\mathring{\mathtt{F}}(c_1,c,w):=x_0\circ \eexp\int_0^1(1+c_1)e_*^{-w(t)X_2}X_1dt\circ e^{cX_2}
	\ee
	on the completion space $\R^2\oplus L^2(I,\R)$ (compare with \eqref{eq:changeofcoord} and \eqref{eq:Gnewcoord}).
	
	Notice that $\lambda\He_0(\mathring{\mathtt{F}})$ is non-degenerate: in fact, since $\gamma$ is rank-two-nice, the intersection of $\ker(d_0\mathring{F})$ with $\R^2\oplus L^2(I,\R)$ is nothing but the subspace $\mathring{P}\oplus \mathring{N}$ introduced in Section~\ref{sec:an_prop}. This allows to establish a ``generalized Morse Lemma'' much in the spirit of \cite[Lemma 1.2]{SaryStab}.
	
	We fix coordinates $z:=(c_1,c,w_2)\in \R^2\oplus L^2(I,\R)$, and we write $\mathring{\mathtt{F}}(z)=\mathring{\mathtt{F}}(0)+d_0\mathring{\mathtt{F}}(z)+d_0^2\mathring{\mathtt{F}}(z)+\mathtt{R}(z)$, where $\mathtt{R}$ denotes a remainder term whose first and second derivatives at zero vanish. 	
	\begin{prop}\label{prop:GML}
		There exist a neighborhood $\mathring{\mc{W}}\subset \R^2\oplus L^2(I,\R)$
		and an origin-preserving diffeomorphism $\mathring{\sigma}:\mathring{\mc{W}}\to \mathring{\mc{W}}$,
		and a diffeomorphism $\mathring{\nu}:\mc{O}_y\to \mc{O}_0$ of neighborhoods $\mc{O}_y\subset M$ and $\mc{O}_0\subset \R\oplus \IM (d_0\mathtt{F})$, respectively of $y$ and $0$, such that for every $z\in \mathring{\mc{W}}$ there holds the identity:
		\be
			\left(\mathring{\nu}\circ \mathring{\mathtt{F}} \circ \mathring{\sigma}\right)(z)=
			\left(\lambda\He_0(\mathring{\mathtt{F}})(z'),d_0\mathring{\mathtt{F}}(z'')\right),
		\ee
		where $z=(z'z'')$ is a coordinate system subordinated to the splitting $\ker(d_0\mathring{\mathtt{F}})\oplus E$.
\end{prop}	
	
	\begin{proof}
		We
		prove in fact a slightly stronger statement: there exist neighborhoods $\mathring{\mc{W}}\subset \R^2\oplus L^2(I,\R)$ and $\mc{O}_y\subset M$, respectively of the origin and of $y$, such that for every $t\in I$ there exist an origin-preserving diffeomorphism $P_t:\mathring{\mc{W}}\to \mathring{\mc{W}}$ and a diffeomorphism $Q_t:\mc{O}_y\to\mc{O}_y$ that preserves $y$, for which the diagram in Figure~\ref{fig:diagram}
	\begin{figure}[tbp]
			\includegraphics[scale=.8]{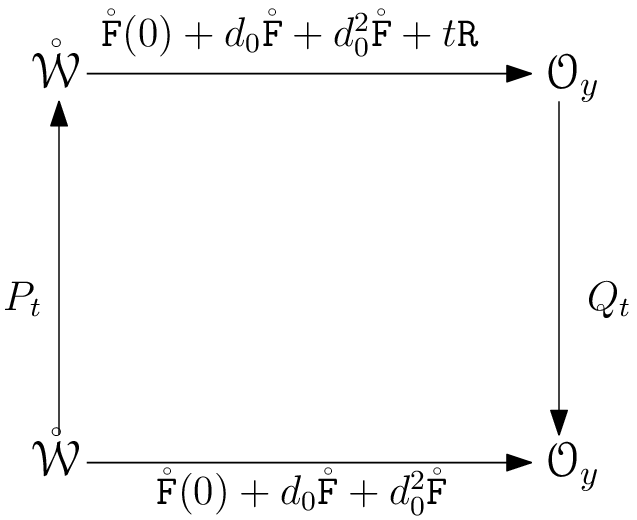}
			\caption{}
			\label{fig:diagram}
		\end{figure}
		commutes.
		
		It is then sufficient to set $\mathring{\sigma}:=P_1$ and $\mathring{\nu}$ to be the composition of $Q_1$ with a diffeomorphism from $\mc{O}_y$ to $\mc{O}_0$ to prove the proposition.
		
		More specifically, as for the classical Moser's trick \cite{Moser}, we look for families $(P_t)_{t\in I}$, $(Q_t)_{t\in I}$ of diffeomorphisms, of the form  $P_t=\eexp\int_0^t X_\tau d\tau$ and $Q_t=\eexpi\int_0^t Y_\tau d\tau$, for suitable locally Lipschitz time-dependent vector fields $X_\tau$ and $Y_\tau$ on $\mathring{\mc{W}}$ and $\mc{O}_y$. For the definition of the right and left chronological exponentials we refer to \cite[Chapter 2]{AgraBook}. For us, it is only important to recall that $\frac{d}{dt}P_t=P_t\circ X_t$ and $\frac{d}{dt}Q_t=Y_t\circ Q_t$.

		Without loss of generality, we assume $\mathring{\mathtt{F}}(0)=0$, and we fix local coordinates on a neighborhood $\mc{O}_y\subset M$ of $y$, subordinated to the splitting $T_{0}M=\mathrm{coker}(d_0\mathring{\mathtt{F}})\oplus \IM(d_0\mathring{\mathtt{F}})$.  We also write mappings $f:\R^2\oplus L^2(I,\R)\to T_0M$ as sums of the form $f=f^\lambda+f^E$, where $f^\lambda=\langle \lambda,f\rangle$ denotes the projection of $f$ along the abnormal direction.

	The commutativity condition expressed by the diagram reads:
	\be
		P_t\circ \big( d_0\mathring{\mathtt{F}}+d_0^2\mathring{\mathtt{F}}+t\mathtt{R} \big)\circ Q_t=d_0\mathring{\mathtt{F}}+d_0^2\mathring{\mathtt{F}}.
	\ee
	Notice that for $t=0$ the identity holds. Differentiating this equation we obtain:
		\be\label{eq:first}
			P_t\circ \left( X_t\cdot\left( d_0\mathring{\mathtt{F}}+d_0^2\mathring{\mathtt{F}}+t\mathtt{R} \right)+\mathtt{R}+\left(d_0\mathring{\mathtt{F}}+d_0^2\mathring{\mathtt{F}}+t\mathtt{R}\right)\circ Y_t
			\right)\circ Q_t=0.
		\ee
		
		We look for solutions to \eqref{eq:first} of the form $X_t=X_t^{\ker}+X_t^E$ and $Y_t=Y_t^{\lambda}$, where $X_t^{\ker}:=\pi_{\ker(d_0\mathring{\mathtt{F}})}X_t$.
		
		 Since $P_t$ and $Q_t$ are diffeomorphisms, solving \eqref{eq:first} is equivalent to solving the following system of equations:
		\be\label{eq:second}
			\left\{\begin{aligned}
				&d_0\mathring{\mathtt{F}} X_t(z)+2d_0^2\mathring{\mathtt{F}}^E\left(z,X_t(z)\right)+ td_z\mathtt{R}^E X_t(z)=-\mathtt{R}^E(z),\\
			&2d_0^2\mathring{\mathtt{F}}^{\lambda}\left(z,X_t(z)\right)+td_z\mathtt{R}^{\lambda} X_t(z)+Y_t^{\lambda}\left(d_0\mathring{\mathtt{F}}(z)+d_0\mathring{\mathtt{F}}^2(z)+t\mathtt{R}(z)\right) =-\mathtt{R}^{\lambda}(z).
			\end{aligned}\right.
		\ee

Let $(d_0\mathring{\mathtt{F}})^{-1}:\IM (d_0\mathring{\mathtt{F}})\to E$ denote the right pseudo-inverse to $d_0\mathring{\mathtt{F}}$.
We solve the first equation of \eqref{eq:second} with respect to $X_t^E$. Given $z\in \R^2\oplus L^2(I,\R)$, let
\be
	\begin{aligned}
		J_t^E(z): E &\to E,\\
		z'&\mapsto z'+2(d_0\mathring{\mathtt{F}})^{-1}d_0^2\mathring{\mathtt{F}}^E(z,z')+t(d_0\mathring{\mathtt{F}})^{-1}d_z\mathtt{R}^E(z').
	\end{aligned}
\ee
Observe that $J_t^E(0)=\mathrm{Id}_E$ for every $t\in I$. Then there exists a neighborhood of the origin $\mathring{\mc{W}}_1\subset \R^2\oplus L^2(I,\R)$ such that $J_t^E(z)$ is invertible for every $z\in\mathring{\mc{W}}_1$, and we have
\be\label{eq:XE}
	X_t^E(z)=-J_t^E(z)^{-1}(d_0\mathring{\mathtt{F}})^{-1}\left(  \mathtt{R}^E(z)+2d_0^2\mathring{\mathtt{F}}^E\left(z,X_t^{\ker}(z)\right)+td_z\mathtt{R}^E X_t^{\ker}(z) \right).
\ee

We substitute $X_t^{E}$ in the second equation of \eqref{eq:second}. To find $X_t^{\ker}$ we solve to the scalar equation:
		\be\label{eq:third}
			\left( A_t(z),X_t^{\ker}(z) \right)_{\R^2\oplus L^2(I,\R)}+Y_t^\lambda\left(d_0\mathring{\mathtt{F}}(z)+d^2_0\mathring{\mathtt{F}}(z)+t\mathtt{R}(z)\right)=-S^\lambda_t(z),
		\ee
		where $(\cdot,\cdot)_{\R^2\oplus L^2(I,\R)}$ denotes the Hilbert product on $\R^2\oplus L^2(I,\R)$, and for every
$t\in I$ we defined the maps
		\be
			\begin{aligned}
			S_t^\lambda:\mathring{\mc{W}}_1&\to \R,\\
			z&\mapsto\lambda \mathtt{R}(z)+2\lambda d_0^2\mathring{\mathtt{F}}\left(z, J_t^E(z)^{-1}(d_0\mathring{\mathtt{F}})^{-1}\mathtt{R}^E(z)  \right)+t\lambda d_z\mathtt{R}\left( J_t^E(z)^{-1}(d_0\mathring{\mathtt{F}})^{-1}\mathtt{R}^E(z) \right),
			\end{aligned}
		\ee
		and $A_t:\mathring{\mc{W}}_1\to \R^2\oplus L^2(I,\R)$ by requiring that:
		\be
			\begin{aligned}
			\left( A_t(z), z' \right)_{\R^2\oplus L^2(I,\R)}&:=2\lambda d_0^2\mathring{\mathtt{F}}(z,z')-2\lambda d_0^2\mathring{\mathtt{F}}\left( z, J_t^E(z)^{-1}(d_0\mathring{\mathtt{F}})^{-1}\left( 2d_0^2\mathring{\mathtt{F}}^E\left(z,z'\right)+td_z\mathtt{R}^E (z') \right) \right)\\ &+t\lambda d_z\mathtt{R}(z')-t\lambda d_z\mathtt{R}\left( J_t^E(z)^{-1}(d_0\mathring{\mathtt{F}})^{-1}\left( 2d_0^2\mathring{\mathtt{F}}^E\left(z,z'\right)+td_z\mathtt{R}^E (z') \right)  \right)
			\end{aligned}
	\ee
	for every $z'\in\R^2\oplus L^2(I,\R)$.
	
	Notice that, for every $t\in I$, the first and the second derivatives of $S_t^\lambda$ at zero vanish. Moreover, $d_0A_t$ is a linear map from $\R^2\oplus L^2(I,\R)$ with the property that
	\be
		\left(d_0A_t(z),z'\right)_{\R^2\oplus L^2(I,\R)}=2\lambda d_0^2\mathring{\mathtt{F}}(z,z').
	\ee
	for every $z,z'\in \R^2\oplus L^2(I,\R)$.
	
	This implies that the linear operator
	\be
		\proj_{\ker(d_0\mathring{\mathtt{F}})}\circ d_0A_t\big|_{\ker(d_0\mathring{\mathtt{F}})}:\ker(d_0\mathring{\mathtt{F}})\to \ker(d_0\mathring{\mathtt{F}})
	\ee
	is in fact the invertible linear operator $L$ associated with the restriction to $\mathring{P}\oplus \mathring{N}$ of $\lambda\He_0\mathring{F}$ of Proposition~\ref{prop:invertible}.
	
	Consider, for $t\in I$,
 	\be\label{eq:Phi}
		\begin{aligned}
			\mathfrak{F}_t:\mathring{\mc{W}}_1&\to \R^2\oplus L^2(I,\R),\\
			z&\mapsto \left(  \pi_{\ker(d_0\mathring{\mathtt{F}})}\circ A_t(z), d_0\mathring{\mathtt{F}}^*\left(d_0\mathring{\mathtt{F}}+d_0^2\mathring{\mathtt{F}}(z)+t\mathtt{R}(z)\right) \right)
		\end{aligned}
	\ee
	where $ d_0\mathring{\mathtt{F}}^*:\IM(d_0\mathring{\mathtt{F}})\to E$ denotes an adjoint operator to $d_0\mathring{\mathtt{F}}$. An easy computation shows that $d_0\mathfrak{F}_t$ has the upper triangular form
		\be\label{eq:jacphi}
			d_0\mathfrak{F}_t=\begin{pmatrix}  L & * \\ 0 & d_0\mathring{\mathtt{F}}^* d_0\mathring{\mathtt{F}} . \end{pmatrix}
		\ee
		If we introduce coordinates $(v,w)$ on $\R^2\oplus L^2(I,\R)$, adapted to the splitting $\ker(d_0\mathring{\mathtt{F}})\oplus E$, 
		this means that $\mathfrak{F}_t(z)=\left( v(z,t), w(z,t) \right)$ defines a local diffeomorphism on some neighborhood $\mathring{\mc{W}}\subset \mathring{\mc{W}}_1$ for every $t\in I$.

 By the Hadamard Lemma there exists a smooth function $\mathfrak{S}_t:\mathring{\mc{W}}\to \ker(d_0\mathring{\mathtt{F}})$ such that
		\be\begin{aligned}\label{eq:fourth}
			S_t^\lambda(z)&=\left(S_t^\lambda\circ \mathfrak{\mathtt{F}}_t^{-1}\right)\left( v(z,t), w(z,t) \right)\\
			                        &=\left( S_t^\lambda\circ \mathfrak{\mathtt{F}}_t^{-1} \right)\left(0,w(z,t)\right)+\left( v(z,t),\mathfrak{S}_t(z) \right)_{\R^2\oplus L^2(I,\R)}.
		\end{aligned}\ee
		
		To conclude the proof it suffices to compare \eqref{eq:third}, \eqref{eq:Phi} and \eqref{eq:fourth}, and to define:
		\be
			X_t^{\ker}:=\mathfrak{S}_t,\ \ Y_t^\lambda:=S_t^\lambda\circ \mathfrak{F}_t^{-1}\circ d_0\mathring{\mathtt{F}}^*
		\ee
		and $X_t^E$ by plugging $X_t^{\ker}$ back in \eqref{eq:XE}. It is not difficult to see that $X_t^{\ker}$ is Lipschitz with vanishing first derivative at the origin, and that also $Y_t^\lambda$ and $X_t^E$ are Lipschitz with respect to their arguments, with vanishing first and second order derivatives at the origin.
	\end{proof}
	
	\subsection{Proof of the main Theorem}\label{sec:proof}
	
	\begin{proof}[Proof of Theorem~\ref{thm:main}]
		We divide the proof into several steps. Let us define
		\be
			\mc{W}:=\mathring{\mc{W}}\cap \left( \R\oplus H^1(I,\R)  \right).
		\ee
		We begin by claiming that the diffeomorphism $\mathring{\sigma}$, provided by Proposition~\ref{prop:GML}, restricts to a diffeomorphism $\sigma:=\mathring{\sigma}\big|_{\mc{W}}$ of $\mc{W}$.
		
		Indeed, it is shown in \cite[Theorem 3.2]{SaryStab} that $\mathring{\sigma}$ has the form of a system of nonlinear Urysohn integral equations of the second kind with small kernels, i.e.
		\be\label{eq:Urysohn}
			\mathring{\sigma}(v)(t)=v(t)-\int_0^1 K(v,\tau,t)d\tau,
		\ee
		where $K$ is differentiable with respect to the $t$-variable. It follows that $\mathring{\sigma}(v)\in  \R\oplus H^1(I,\R)$ if and only if $v\in  \R\oplus H^1(I,\R)$. The claim follows.
		
		Next, let $L^\infty_Z(I,\R)$ be defined as in \eqref{eq:Banspacezero}, and consider the projection
		\be
			\pi_{\R\oplus H^1(I,\R)}:\R\oplus L^\infty_Z(I,\R)\oplus H^1(I,\R) \to \R\oplus H^1(I,\R),
		\ee
		where the decompositions are orthogonal with respect to the $L^2(I,\R)$-product.
		
		If we now define
		\be\label{eq:mcV}
			\begin{aligned}
			\mc{W}_{\rho}&:=\pi_{\R\oplus H^1(I,\R)}^{-1}(\mc{W})\cap \rho(\mc{V}_2),\\
			\mc{V}&:=\rho^{-1}(\mc{W}_{\rho}),
			\end{aligned}
		\ee
		we conclude by Proposition~\ref{prop:contrep} that $\mc{V}$ is an open neighborhood of the origin in $L^\infty(I,\R)\oplus H^1(I,\R)$ contained in $\mc{V}_2$ (and an $(\infty,2)$-neighborhood, in fact).
		
		Lastly, let us observe that
		\be\label{eq:varphirep}
			\varphi:=\left( \sigma\big|_{\R\oplus \{0\}},\mathrm{Id}_{L^\infty_Z(I,\R)},  \sigma\big|_{\{0\}\oplus H^1(I,
			R)}  \right)
		\ee
		provides the desired change of coordinates on $\rho(\mc{V})=\mc{W}_{\rho}$. Then the theorem follows with $\mc{V}$ as in \eqref{eq:mcV}, $\mu:=\varphi\circ \rho$, $\mc{V}':=\mu(\mc{V})$ and $\psi:=\mathring{\nu}$, $\mc{O}_y$, $\mc{O}_0$ as in Proposition~\ref{prop:GML}.
	\end{proof}

	\subsection{Isolation of rank-two-nice singular curves}\label{sec:isolation} We prove in this section some isolation properties of rank-two-nice singular curves in $\Omega(y)$ among the class of extremal curves, i.e. among critical points of the extended endpoint map $\mc{F}$ (see \eqref{eq:PMP}).
	
	Let $\gamma$ be a reference rank-two-nice singular curve, and fix a local chart around $\gamma$.

	\begin{proof}[Proof of Theorem~\ref{thm:isolation}]
	Theorem~\ref{thm:main} yields local coordinate systems on an $(\infty,2)$-neighborhood of the origin $\mc{V}\subset \ker(d_0F)\oplus \IM(d_0F)$ and on a neighborhood $\mc{O}_0\subset M$ of the origin such that
		\be\label{eq:Fcoord}
			F(v,w)=\begin{pmatrix} \lambda\He_0F(v) \\ w \end{pmatrix}
		\ee
		for every $(v,w)\in\mc{V}$. In particular, if $(v,w)\in\Omega(y)$ then it must be the case that $w=0$. This proves the inclusion
		\be
			\Omega(y)\cap\mc{V}\subset\left\{ (v,0)\mid v\in \ker(d_0F)  \right\}.
		\ee
		Let us write every $v\in \ker(d_0F)$ as a pair $v=(v_1,v_2)$, in subordination to the splitting $L^\infty(I,\R)\oplus L^2(I,\R)$.
		
		Now we claim the following: there exists a neighborhood of the origin $\mc{U}\subset \mc{V}$ such that, if $(v,w)\in \mc{U}$ is associated with an extremal curve, then $v_2\equiv 0$ a.e. on $I$.
		
		Indeed, it is sufficient to consider any neighborhood $\mc{U}\subset \mc{V}$ with the following property: for every element $(v,w)\in \mc{U}$, $d_{(v,w)}J:L^2(I,\R^2)\to \R$ is not the zero map. Notice that this is certainly possible since $\nabla_0J=(1,0)$.
		
		Fix an extremal control $(v,w)\in \mc{U}$, and let $a\in C^\infty(M)$ be such that $a(0)=0$ and $d_0a=\lambda$. We consider the composition $a\circ F: \mc{V}\to \R$. In coordinates, we have $(a\circ F)(v,w)=\lambda\He_0F(v)$, whence
		\be
			d_{(v,w)}(a\circ F)(x,y)=2\lambda \He_0F(v,x)
		\ee
		for every $(x,y)\in
		\BS$.
		
		Assuming by contradiction $v_2$ different from zero, we conclude that there exists $x\in \ker(d_0F)$ such that $\lambda\He_0F(v,x)\ne 0$, in fact $\lambda\He_0F(v,\cdot)$ is not the zero operator since $\gamma$ is rank-two-nice. However, this implies that $d_{(v,w)}\mc{F}$ has maximal rank, and a covector $\xi$ satisfying \eqref{eq:PMP} cannot exist, providing the absurd.
		
		So far, we have shown the following: there exists a neighborhood of the origin $\mc{U}\subset \BS$, such that if $u=(v,w)$ is a control associated with an extremal curve in $\Omega(y)$, then $w\equiv v_2\equiv0$ a.e. on $I$, implying that $u_2\equiv 0$ a.e. on $I$ and, in turn, $u=(u_1,0)$.  On the other hand, from the formula
		\be
			y=F(u_1,0)=x_0\circ e^{(1+\int_0^1 u_1(t)dt)X_1}=y\circ e^{\int_0^1 u_1(t)dtX_1}
		\ee
		we see that $(u_1,0)\in \Omega(y)$ if and only if $u_1$ is of zero mean.
		
		Finally, if we choose $\mc{U}$ so that $1+\int_0^t u_1(t)dt$ is of constant sign for every $u\in \mc{U}$, we conclude that the curve
		\be
			t\mapsto F_t(u_1,0)= x_0\circ e^{(1+\int_0^t u_1(\tau)d\tau)X_1}
		\ee
		is just a reparametrization of $\gamma$, and the theorem follows.
	\end{proof}
	
	\begin{cor}\label{cor:isol}
		There exists a weak neighborhood of the origin $\mc{V}_4\subset L^2(I,\R^2)$, such that the only singular controls contained in $\mc{V}_4$ are of the form $v=(v_1,0)$, with $v_1$ of zero mean.
	\end{cor}
	
	\begin{proof}
		Let $(v_n)_{n\in \N}\subset \Omega(y)$ be a sequence of singular controls weakly converging to zero in $L^2(I,\R^2)$. By Proposition~\ref{prop:contendp}, $F_t(\gamma_{v_n})\to F_t(\gamma)$ and $d_{v_n}F_t\to d_0F_t$ uniformly on $t\in I$.
		
		Since the conditions in Definition~\ref{defi:ranktwonice} describing a rank-two-nice curve are open conditions among singular curves, we assume that $\gamma_{v_n}$ is rank-two-nice for every $n\in \N$. In particular the controls $(v_n)_{n\in \N}$ are smooth (hence contained in $\BS$).
		
		For every $n\in \N$, pick $\lambda_{v_n}\in \IM(d_{v_n}F)^\perp$ (resp. $\lambda_v\in \IM(d_vF)^\perp$), and define
		\be
			\lambda^0_{v_n}:=\frac{\left( d_{v_n}F \right)^*\lambda_{v_n}}{\|\left( d_{v_n}F \right)^*\lambda_{v_n}\|}\in T_{x_0}^*M,\ \ \lambda^0_v:=\frac{\left( d_vF \right)^*\lambda_v}{\|\left( d_vF \right)^*\lambda_v\|}\in T_{x_0}^*M,
		\ee
		where $\|\cdot\|$ denotes some given norm on $T^*_{x_0}M$. Clearly, $\lambda^0_{v_n}\to \lambda^0_v$, and a similar argument shows that
		\be
			\lambda_{v_n}(t):=(d_{v_n}F_t^{-1})^*\lambda_{v_n}^0 \to  (d_vF_t^{-1})^*\lambda_v^0=:\lambda_v(t)
		\ee
		uniformly on $I$, as $n\to +\infty$ (recall that $(d_{w}F_t^{-1})^*:T_{x_0}^*M\to T^*_{F_t(w)}M$ for every $w\in L^2(I,\R^2)$).
		
		From the uniform convergence of the associated biextremals, we conclude that the convergence of $v_n$ to $v$ takes place in $\BS$. Then Theorem~\ref{thm:isolation} applies, and allows to conclude.
	\end{proof}

	\section{Jacobi fields and computations of conjugate points}\label{sec:examples} We explain in this last chapter how to characterize conjugate points along a rank-two-nice singular curve $\gamma$. We suppose for simplicity that $\Delta_x=\textrm{span}\{X_1(x),X_2(x)\}$ in the domain under consideration and we consider, as in Definition~\ref{defi:locchar}, a local chart $\mc{V}_1\in L^2(I,\R^2)$ centered at zero.
	
	For every $s\in I$ we define $\mc{V}_1^s\subset L^2([0,s],\R^2)$ to be the restriction of $\mc{V}_1$ to the subinterval $[0,s]\subset I$, and if $\gamma$ is an admissible curve, $\gamma_s:=\gamma\big|_{[0,s]}$ denotes the restriction of $\gamma$ to $[0,s]$. The endpoint map, the energy and the extended endpoint map have natural restrictions to $\mc{V}_1^s$, which we denote by $F_s,J_s$ and $\mc{F}_s$, respectively. In particular
	$F_s(\gamma)=F(\gamma_s)$, $J_s(\gamma)=J(\gamma_s)$ and $\mc{F}_s(\gamma)=\mc{F}(\gamma_s)$.
	
	Fix a rank-two-nice curve $\gamma$. To detect conjugate times it is convenient to work on the cotangent space $T^*M$. We call $t\mapsto \lambda(t)$ the extremal lift of $\gamma$, normalized so that $\lambda_0:=\lambda(0)\in T^*_{x_0}M$ has norm one. Then $\lambda(t)=(e^{-tX_1})^*\lambda_0=(e^{(s-t)X_1})^*\lambda(s)$ on $[0,s]$.

	Recall that, by \eqref{eq:kerG} and Corollary~\ref{cor:dec1}, we have
	\be\label{eq:kers}
		\begin{aligned}
	\ker (d_{0}\mathring{F}_s)&=\left\{ (v,c,w)\mid\int_0^sw(t)\dot{g}_t^s(\gamma(s))dt = \int_0^s v(t)dt X_1(\gamma(s)) + c X_2(\gamma(s))\right\},\\ &\subset L^2([0,s],\R) \oplus \R\oplus L^2([0,s],\R)\\
	\ker (d_{0}\mathring{\mc{F}}_s)&=\left\{ (v,c,w)\in \ker (d_{0}\mathring{F}_s) \mid \int_0^sv(t)dt=0  \right\},
		\end{aligned}
	\ee
	and that on $L^2(I,\R)\oplus \R\oplus L^2(I,\R)$ the Hessian map is given by
	\be\label{eq:quadr}
	\int_0^s\langle \lambda(s),[ \dot{g_t^s},g_t^s ](\gamma(s))\rangle w(t)^2dt+\int_0^s\left\langle \lambda(s),\left[cX_2+\int_0^t w(\tau)\dot{g_\tau^s}d\tau,w(t)\dot{g_t^s} \right](\gamma(s)) \right\rangle dt,
	\ee
	both for $F_s$ and $\mc{F}_s$.
	
	\subsection{The symplectic framework}
	
	Let $\sigma\in \Lambda^1(T^*M)$ be the standard Liouville one-form and $\omega:=d\sigma\in \Lambda^2(T^*M)$ denote the canonical symplectic form on $T^*M$. We introduce the following notations:
	\begin{itemize}
		\item For every $\mu\in T^*M$ and $x\in T_\mu(T^*M)$, we denote the skew-orthogonal complement to $x$ in $T_\mu(T^*M)$ by $
			x^\angle:=\left\{ y\in T_\mu(T^*M)\mid \omega_\mu(x,y)=0 \right\}$.
		\item For every subspace $N\subset TT^*M$, we denote the skew-orthogonal complement to $N$ by $N^\angle:=\left\{ y\in TT^*M\mid \omega(x,y)=0,\,\text{for every }x\in N  \right\}$.
		\item For every Hamiltonian function $a\in C^\infty(T^*M)$, we denote by $\vec{a}\in\Vc(T^*M)$ its Hamiltonian lift, given by the formula
		\be
			\omega_\mu(\cdot,\vec{a})=d_\mu a(\cdot),\ \ \text{for every }\mu\in T^*M.
		\ee
	\end{itemize}
	
	Fix $s\in I$, and let $\pi:T^*M\to M$ be the canonical bundle projection. We introduce the fiberwise linear Hamiltonians
	\be\label{eq:hamilt_linear}
			\begin{aligned}
				\xi_1:T^*M&\to \R,\\
					\mu&\mapsto\left\langle \mu,X_1(\pi(\mu)) \right\rangle,
			\end{aligned} \ \ \ \ \ \ \ \
			\begin{aligned}
				\eta_t^s:T^*M&\to \R,\\
					\mu&\mapsto\left\langle \mu,g_t^s(\pi(\mu)) \right\rangle,
			\end{aligned}
	\ee
	where we set $g^s_t:=e^{(s-t)X_1}_*X_2$ for every $t\in [0,s]$. Accordingly, we define the Hamiltonian lifts $\vec{\xi}_1$ and $\vec{\eta}_t^s$, and we observe that $\vec{\eta}_s^s=\vec{\xi}_2$ for every $s\in I$, where $\xi_2$ is the Hamiltonian function associated with $X_2$ as in \eqref{eq:hamilt_linear}. 
	
	\noindent We identify $T^*_{\gamma(s)}M$ with $T_{\lambda(s)}(T^*_{\gamma(s)}M)$: in terms of local bases we identify $\nu=\sum_{i=1}^m\nu_idx_i$ and $\nu=\sum_{i=1}^m \nu_i\partial_{\lambda_i(s)}$. Then $\lambda(s)$ is associated with the value at $\lambda(s)$ of the Euler vector field $\mathbf{e}\in \Vc(T^*M)$:
	\be
		\mathbf{e}(\lambda(s))=\lambda_1(s)\partial_{\lambda_1(s)}+\dots+\lambda_m(s)\partial_{\lambda_m(s)}.
	\ee

	We define
	\begin{itemize}
		\item $\Sigma:=\{ \nu\in T_{\lambda(s)}(T^*M)\mid \omega_{\lambda(s)}(\lambda(s),\nu)=0 \}/\R\lambda(s)$ to be the skew-orthogonal complement of $\lambda(s)$ in the symplectic space $T_{\lambda(s)}(T^*M)$, factorized by $\mathrm{span}\{\lambda(s)\}$. Notice that $\Sigma$ is a symplectic subspace of $T_{\lambda(s)}(T^*M)$, of dimension $2(m-1)$.
		\item $\Pi:=T^*_{\gamma(s)}M/\R \lambda(s)$. 	\end{itemize}
	 Our identifications allow to identify $\Pi$ with a Lagrangian (i.e. of dimension $m-1$) subspace of $\Sigma$,
and for every $\mu\in T^*_{\gamma(s)}M$ we have:
	\be
			\omega_{\lambda(s)}(\mu,\vec{\xi}_1):=\left\langle \mu, X_1(\gamma(s)) \right\rangle,\ \ \omega_{\lambda(s)}(\mu,\vec{\eta}_t^s):=\langle \mu, g_t^s(\gamma(s)) \rangle
	\ee
	where $\langle\cdot,\cdot\rangle$ denotes the dual product. In particular $\omega_{\lambda(s)}(\lambda(s),\vec{\xi}_1)=\omega_{\lambda(s)}(\lambda(s),\vec{\eta}_t^s)=0$ for every $t\in [0,s]$, and thus both $\vec{\xi}_1$ and $\vec{\eta}_t^s$ can be identified with elements of $\Sigma$.
	
	Moreover, by the equality $\omega_{\lambda(s)}(\vec{\xi}_1,\vec{\eta}_t^s)=\langle\lambda(s),[X_1,g_t^s](\gamma(s))\rangle$, we deduce from \eqref{eq:condperp2}  that $\omega_{\lambda(s)}(\vec{\xi}_1,\vec{\eta}_t^s)\equiv 0$ for every $t\in [0,s]$, obtaining as well
	\be\label{eq:newrel_zero}
		\omega_{\lambda(s)}(\vec{\xi}_1,\dot{\vec{\eta}}_t^s)\equiv 0,\ \ \text{for every }t\in [0,s].
	\ee

	Let $l_{t,s}^0:=\omega_{\lambda(s)}(\dot{\vec{\eta}}_t^s,\vec{\eta}_t^s)=\left\langle \lambda(s),[\dot{g}_t^s,g_t^s](\gamma(s)) \right\rangle$. The quadratic form \eqref{eq:quadr} becomes
	\be\label{eq:quadrs}
		\int_0^s l_{t,s}^0 w(t)^2dt+\int_0^s \omega_{\lambda(s)}\left( c\vec{\eta}_s^s+\int_0^s w(\tau)\dot{\vec{\eta}}_\tau^sd\tau, w(t)\dot{\vec{\eta}}_t^s \right)dt,
	\ee
	while the kernel conditions \eqref{eq:kers} read
	\be\label{eq:kerss}
		\begin{aligned}
	\ker (d_{0}\mathring{F}_s)&=\left\{ (v,c,w)\mid\int_0^sw(t)\dot{\vec{\eta}}_t^s(\lambda(s))dt = \int_0^s v(t)dt \vec{\xi}_1(\lambda(s)) + c
	\vec{\xi}_2(\lambda(s))\right\},\\
	\ker (d_{0}\mathring{\mc{F}}_s)&=\left\{ (v,c,w)\mid\int_0^sw(t)\dot{\vec{\eta}}_t^s(\lambda(s))dt = c
	\vec{\xi}_2(\lambda(s)),\; \int_0^s v(t)dt \vec{\xi}_1(\lambda(s)) = 0 \right\}.
		\end{aligned}
	\ee

	These conditions are equivalent to the following:
	\begin{itemize}
		\item [(a)] $w\in L^2([0,s],\R)$ belongs to $\ker (d_{0}\mathring{F}_s)$ if, for every $\nu\in \Pi\cap\mathrm{span}\{ \vec{\xi}_1(\lambda(s)),
		\vec{\xi}_2(\lambda(s)) \}^\angle+ \R\vec{\xi}_1(\lambda(s))$, we have $\int_0^s\omega_{\lambda(s)}(\nu,\dot{\vec{\eta}}_t^s)w(t)dt=0$. Notice that $\nu$ can be defined modulo $\vec{\xi}_1(\lambda(s))$ because of \eqref{eq:newrel_zero}.
		\item [(b)] $w\in L^2([0,s],\R)$ belongs to $\ker (d_{0}\mathring{\mc{F}}_s)$ if, for every $\nu\in \Pi\cap\mathrm{span}\{
		\vec{\xi}_2(\lambda(s)) \}^\angle$, we have $\int_0^s\omega_{\lambda(s)}(\nu,\dot{\vec{\eta}}_t^s)w(t)dt=0$.
	\end{itemize}

	\subsection{Conjugate points} Let $s\in I$. We recall that $\gamma(s)$ is a conjugate point along $\gamma$
	if the kernel of the quadratic form in \eqref{eq:quadrs}, whose domain is determined by either one of the two conditions in \eqref{eq:kerss},
	admits a nontrivial orthogonal decomposition of the form
	\be
		\left\{ (v,0,0)\mid \int_0^1 v(t)dt=0 \right\}\oplus Z_1' .
	\ee
	The multiplicity of $\gamma(s)$ as a conjugate point equals the dimension of $Z_1'$ (compare with Definition~\ref{defi:conj}).
	
	Thus
	$\gamma(s)$ is a conjugate point along $\gamma$ if and only if there exists $(v_0,c_0,w_0)\in \ker(d_{0}\mathring{F}_s)$ (resp. $(v_0,c_0,w_0)\in\ker(d_{0}\mathring{\mc{F}}_s)$) such that
	\be\label{eq:kercond}
		\begin{aligned}
		\mathrm{(i)} & & & (c_0,w_0)\ne (0,0),\\
		\mathrm{(ii)} & & & \int_0^s\left(  l_{t,s}^0 w_0(t)+\omega_{\lambda(s)}\left(c_0
		\vec{\xi}_2+\int_0^tw_0(\tau)\dot{\vec{\eta}}		_\tau^sd\tau,\dot{\vec{\eta}}_t^s    \right)\right)w(t)dt=0
		\end{aligned}	
	\ee
	for every $(v,c,w)\in\ker (d_0\mathring{F}_s)$ (resp. for every $(v,c,w)\in\ker (d_0\mathring{\mc{F}}_s)$).
	
	By \eqref{eq:kerss}, this implies that for all $t\in [0,s]$
	\be\label{eq:Je}
		l_{t,s}^0 w_0(t)+\omega_{\lambda(s)}(c_0
		\vec{\xi}_2+\int_0^tw_0(\tau)\dot{\vec{\eta}}_\tau^sd\tau,\dot{\vec{\eta}}_t^s    )=\omega_{\lambda(s)}(-\nu,\dot{\vec{\eta}}_t^s),
	\ee
	for some $\nu\in\Pi\cap\mathrm{span}\{ \vec{\xi}_1(\lambda(s)),
	\vec{\xi}_2(\lambda(s)) \}^\angle+\R\vec{\xi}_1(\lambda(s))$ (resp. $\nu\in\Pi\cap\mathrm{span}\{
	\vec{\xi}_2(\lambda(s)) \}^\angle
$).
	
	Let
		\be
			k(t):=\int_0^t w_0(\tau)\dot{\vec{\eta}}_\tau^s(\lambda(s))d\tau+c_0
			\vec{\xi}_2(\lambda(s))+\nu.
		\ee
	Multiplying both its sides by $\dot{\vec{\eta}}_t^s(\lambda(s))$, we rewrite \eqref{eq:Je} as the Jacobi equation on $\Sigma$:
	\be\label{eq:Jes}
		l_{t,s}^0\dot{k}(t)=\omega_{\lambda(s)}(\dot{\vec{\eta}}_t^s,k(t))\dot{\vec{\eta}}_t^s(\lambda(s)),
	\ee
	where the corresponding boundary conditions become, respectively,
	\be\label{eq:bc}
		\begin{aligned}
	(a)&\ \ \left\{
	\begin{aligned}
		k(0) & \in \Pi\cap\mathrm{span}\{ \vec{\xi}_1(\lambda(s)),
		\vec{\xi}_2(\lambda(s)) \}^\angle+\mathrm{span}\{\vec{\xi}_1(\lambda(s)),
		\vec{\xi}_2(\lambda(s))\},\\
		k(s) & \in \Pi
	\end{aligned}
	\right.\\
	(b)&\ \
	\left\{
	\begin{aligned}
		k(0) & \in \Pi\cap\mathrm{span}\{
		\vec{\xi}_2(\lambda(s)) \}^\angle+\mathrm{span}\{
		\vec{\xi}_2(\lambda(s))\},\\
		k(s) & \in \Pi
	\end{aligned}
	\right.
		\end{aligned}
	\ee
	
	\begin{prop}\label{prop:conj_equiv}
		Let $s\in I$. The point $\gamma(s)$ is a conjugate point along $\gamma$ for the map $F_s$ (resp. for $\mc{F}_s$) if and only if there exists a solution $k(\cdot)$ to \eqref{eq:Jes}, that satisfies an appropriate boundary condition in \eqref{eq:bc}.
	\end{prop}
	
	\begin{remark}\label{rem:sols}
Observe that in both cases $k(0)$ belongs to a Lagrangian subspace of $\Sigma$, of the form $\Pi\cap \Gamma^\angle+\Gamma$ with $\Gamma$ isotropic.
	\end{remark}

	\subsection{Regular distributions}
	
	We specify now the Jacobi equations for a particular class of rank-two sub-Riemannian structures, the so-called regular structures, that have been investigated e.g. in \cite{Sus4d,SLShortest}. Assume $M$ is an $(m+2)$-dimensional manifold, and that $\Delta=\mathrm{span}\{ X_1,X_2 \}$ in the domain under consideration. Moreover, we assume that:
	\begin{itemize}
		\item  $X_1,X_2,\dots,(\mathrm{ad} X_1)^{m-1}X_2$ are linearly independent.
		\item There exist smooth functions $\beta,\{ \alpha^i,i=0,\dots,m-1\}$ on $M$, such that
		\be\label{eq:dep}
		(\mathrm{ad}X_1)^mX_2=\beta X_1+\sum_{i=0}^{m-1}\alpha^i(\mathrm{ad}X_1)^iX_2.
		\ee
		\item  $[[X_1,X_2],X_2]$ is linearly independent from $V=\mathrm{span}\{ X_1,X_2,\dots, (\mathrm{ad}X_1)^{m-1}X_2 \}$.
	\end{itemize}
	Under these hypotheses it turns out \cite[Section 8]{Sus4d} that integral curves of the vector field $X_1$ are indeed corank-one abnormal geodesics for $\Delta$. Moreover, these curves are also strictly abnormal as soon as $\beta\neq 0$ along the trajectory.
	
	With $g_t^s$ as in \eqref{eq:hamilt_linear} and $i\in\N$, we define
	\be\label{eq:lis}
		\begin{aligned}
		g_t^{s,(i)}&:=\partial_t^{(i)}g_t^s=(-1)^ie^{(s-t)X_1}_*(\mathrm{ad}X_1)^iX_2,\\
		l^{(i)}_{t,s}&:=\langle\lambda(s),[g_t^{s,(1)},g_t^{s,(i)}](\gamma(s))\rangle=\omega_{\lambda(s)}(\vec{\eta}_t^{s,(1)},\vec{\eta}_t^{s,(i)}).
		\end{aligned}
	\ee

	Calling $\beta_t:=\beta(\gamma(t))$ and $\alpha_t^i:=\alpha^i(\gamma(t))$, it is immediate to deduce from \eqref{eq:dep} its symplectic version along $\gamma$, that is
	\be\label{eq:deps}
		\vec{\eta}_{t}^{s,(m)}(\lambda(s))=\beta_t\vec{\xi}_1(\lambda(s))+\sum_{i=0}^{m-1}\alpha_t^i\vec{\eta}_t^{s,(i)}(\lambda(s)).
	\ee
Let $Z:=\{\vec{\xi}_1(\lambda(s)), \vec{\eta}_t^s(\lambda(s)),t\in [0,s] \}$, and notice that we have the decomposition $\Sigma=\Pi\oplus Z$. For every $\tau\in [0,s]$, \eqref{eq:deps}, yields:
\be
	Z=\mathrm{span}\left\{ \vec{\xi}_1(\lambda(s)), \vec{\eta}_\tau^s(\lambda(s)),\dots, \vec{\eta}_\tau^{s,(m-1)}(\lambda(s)) \right\}.
\ee
Notice that $Z$ is not a Lagrangian subspace, nonetheless the symplectic form $\omega_{\lambda(s)}$ defines a non-degenerate splitting between $\Pi$ and $Z$.

Let us write $k(t)=z_t+\theta_t$, with $z_t\in Z$ and $\theta_t\in \Pi$. Then \eqref{eq:Jes} splits as the differential system of equations:
	\be\label{eq:Jess}
		\left\{\begin{aligned}
		l_{t,s}^0\dot{z}_t&=\omega_{\lambda(s)}(\vec{\eta}_t^{s,(1)},z_t+\theta_t)\vec{\eta}_t^{s,(1)}(\lambda(s)),\\
		\dot{\theta}_t&=0,
		\end{aligned}\right.
	\ee
	whose boundary conditions are given respectively by (compare with \eqref{eq:bc}):
	\be\label{eq:bc_coord}
		\begin{aligned}
					(a) && z_s&=0, && z_0\in \mathrm{span}\{ \vec{\xi}_1(\lambda(s)),
		\vec{\xi}_2(\lambda(s)) \}, && \omega_{\lambda(s)}(
		\vec{\xi}_2,\theta_0)=\omega_{\lambda(s)}(\vec{\xi}_1,\theta_0)=0,\\
		(b) && z_s&=0, && z_0\in \mathrm{span}\{
		\vec{\xi}_2(\lambda(s)) \}, && \omega_{\lambda(s)}(
		\vec{\xi}_2,\theta_0)=0.
		\end{aligned}	
	\ee
	
	If we write $z_t=z_t^f\vec{\xi}_1(\lambda(s))+\sum_{i=0}^{m-1}z_t^i\vec{\eta}_t^{s,(i)}(\lambda(s))$ and we define $\zeta_t:=\omega_{\lambda(s)}(\vec{\eta}_t^s,\theta_0)$ and $l_{t,s}^{(i)}$ as in \eqref{eq:lis}, we see that \eqref{eq:Jess} is equivalent to the following system of equations:
	\be\label{eq:sysje}
	\left\{
	\begin{aligned}
		\dot{z}_t^f & =  -\beta_t z_t^{m-1}, &  z_
		s^f=0,\\
		\dot{z}_t^0 & =  -\alpha_t^0 z_t^{m-1}, & z_
		s^0=0, \\
		l_{t,s}^0(\dot{z}_t^1+\alpha_t^1 z_t^{m-1})& =  \sum_{j=2}^{m-1}l_{t,s}^{(j)}z_t^j+\dot{\zeta}_t,\ & z_
		s^1=0,\\
		\dot{z}_t^j+\alpha_t^jz_t^{m-1}& =  -z_t^{j-1},\ \ \text{for every } j=2,\dots,m-1, & z_
		s^j=0,\\
		\zeta_t^{(m)}& = \beta_t \omega_{\lambda(s)}(\vec{\xi}_1,\theta_0)+\sum_{i=0}^{m-1}\alpha_t^i\zeta_t^{(i)}, & \zeta_s=0,
	\end{aligned}
	\right.
	\ee
	and $\gamma(s)$ is a conjugate point along $\gamma$ for $F_s$ (resp. for $\mc{F}_s)$ if and only if \eqref{eq:sysje} admits a nontrivial solution that further satisfies the boundary conditions \eqref{eq:bc_coord}
	\begin{itemize}
		\item [(a)] $z_
		0^i=0$ for all $i=1,\dots,m-1$,
		\item [(b)] $z_
		0^i=0$ for all $i=1,\dots,m-1$ and also $z_
		0^f=0$.
	\end{itemize}
	
	\subsection{The Engel case}\label{sec:fourdim}
	
	We study conjugate points along rank-two-nice singular curves on Engel structures, that is regular rank-two sub-Riemannian structures $(M,\Delta)$ with $M$ of dimension $4$. In this case, the vector field $X_1$ satisfies all the conditions on the Lie brackets required in the previous part by the results in \cite{Ger}.

	We begin with the endpoint map $F_s$. Taking into account that $\omega_{\lambda(s)}(\vec{\xi}_1,\theta_0)=0$, the relevant equations to solve \eqref{eq:sysje} are:
	\be\label{eq:sysjef}
	\left\{
	\begin{aligned}
		l_{t,s}^0(\dot{z}_t^1+\alpha_t^1 z_t^{1})& =  \dot{\zeta}_t, & z_s^1=0,\\
		\ddot{\zeta}_t& = \alpha_t^0\zeta_t+\alpha_t^1\dot{\zeta}_t, & \zeta_s=0.
	\end{aligned}
	\right.
	\ee
	
	We see from \eqref{eq:deps} that
	\be\label{eq:dlis}
		\dot{l}_{t,s}^0=\frac{d}{dt}\omega_{\lambda(s)}(\dot{\vec{\eta}}_t^s,\vec{\eta}_t^s)=\omega_{\lambda(s)}(\ddot{\vec{\eta}}_t^s,\vec{\eta}_t^s)=\alpha_t^1l_{t,s}^0,
	\ee
	yielding that $l_{s,t}^0=l_{0,s}^0e^{\int_0^t\alpha_\tau^1d\tau}$, and therefore
	\be\label{eq:ddlis}
		l_{0,s}^s\frac{d}{dt}\left(z_t^1e^{-\int_0^t\alpha_\tau^1d\tau}\right)=\dot{\zeta}_t.
	\ee
	Because $\gamma$ is rank-two-nice, $l_{0,s}^0\ne 0$ and then the further requirement $z_0^1=0$ is met if and only if $\zeta_0=\omega_{\lambda(s)}(\vec{\eta}_0^s,\theta_0)=0$. So far, we have thus established the relations
	\be
		\omega_{\lambda(s)}(\vec{\eta}_s^s,\theta_0)=\omega_{\lambda(s)}(\vec{\eta}_0^s,\theta_0)=\omega_{\lambda(s)}(\vec{\xi}_1,\theta_0)=0.
	\ee
	
	On the other hand $\theta_0\in \Pi$ belongs to a three dimensional space, and then the linear map $\omega_{\lambda(s)}(\cdot,\theta_0)
	$ cannot have a three dimensional kernel, for otherwise $\theta_0$ would be zero. This implies that $\gamma(s)\in I$ is a conjugate point along $\gamma$ if and only if
	\be
		X_1(\gamma(s))\wedge X_2(\gamma(s))\wedge g_0^s(\gamma(s))=0.
	\ee
	
	Similar computations hold for the extended endpoint map $\mc{F}_s$. Here, the relevant equations are
	\be\label{eq:sysjes}
	\left\{
	\begin{aligned}
		\dot{z}_t^f & =  -\beta_t z_t^1, & z_s^f=0,\\
		l_{t,s}^0(\dot{z}_t^1+\alpha_t^1 z_t^{1})& =  \dot{\zeta}_t, &  z_s^1=0,\\
		\ddot{\zeta}_t& = \alpha_t^0\zeta_t+\alpha_t^1\dot{\zeta}_t+\beta_t\omega_{\lambda(s)}(\vec{\xi}_1,\theta_0), & \zeta_s=0.
	\end{aligned}
	\right.
	\ee
	
	In addition to the equations found before, we require $z_0^f=0$. Since $z_0^1=\zeta_0$, we obtain from \eqref{eq:ddlis} that $z_t^1=\frac{1}{l_{0,s}^0}\zeta_t e^{-\int_0^t \alpha_\tau^1d\tau}$, which readily yields $\int_0^s\beta_t \zeta_t e^{-\int_0^t\alpha_r^1dr}dt=0$.
	
	In terms of vector fields we conclude that $\gamma(s)$ is a conjugate time along $\gamma$ if and only if
	\be
		X_2(\gamma(s))\wedge g_0^s(\gamma(s))\wedge \int_0^s \beta_t  e^{-\int_0^t\alpha_\tau^1d\tau}g_t^s(\gamma(s))dt=0.
	\ee
		
	We are ready to complete the discussion of the example of Section~\ref{sec:example}. From the structural equations we see that
	\be
		[X_1,[X_1,X_2]]=\frac{X_1}{2}-X_2,
	\ee
	whence $\alpha_t^0\equiv -1$, $\alpha_t^1\equiv 0$ and $\beta_t\equiv 1/2$. Moreover it is not restrictive to assume by \eqref{eq:dlis} that  $l_{t,s}^0\equiv 1$, and we suppose as well $\dot{\zeta}_0=1$.
	
	In the case of the endpoint map $F_s$ it is immediate to solve for $\zeta$, and we find $\zeta_t=\sin(t)$. A time $s\in I$ yields a conjugate point $\gamma(s)$ along $\gamma$ if and only if $\zeta_s=0$, whence the computation of $a_F$ readily follows.
	
	We pass to $\mc{F}_s$, and we denote for simplicity $\ell:=\omega_{\lambda(s)}(\vec{\xi}_1,\theta_0)\ne 0$. Solving for $\zeta_t$ and $z_t^f$ (imposing $z_0^f=0)$ yields:
	\be
		\begin{aligned}
			\zeta_t&=\frac{1}{2}\left( \ell-\ell \cos(t)+2\sin(t) \right),\\
			z_t^f&=\int_0^t\frac{1}{2}\zeta_\tau d\tau=\frac{1}{4}\left( 2+\ell t-2\cos(t)-\ell \sin(t) \right).
		\end{aligned}
	\ee
	We look for solutions to $\zeta_t=z_t^f=0$. Clearing out $\ell$, we see that $\gamma(s)$ is a conjugate point if and only if $s$ solves
	\be
		\frac{2\cos(s)-2+s\sin(s)}{2\cos(s)-2}=0.
	\ee
	Now, notice that if $\cos(s)=1$, $z_s^f=0$ if and only if $s=0$ or $\ell=0$, which is not possible by assumption. Then we may forget about the denominator and arrive to the claimed expression for $a_\mc{F}$.

	\bibliographystyle{abbrv}
	\bibliography{Biblio}

\begin{thebibliography}{10}

\bibitem{Ag_Open}
A.~A. Agrachev.
\newblock Some open problems.
\newblock In {\em Geometric control theory and sub-{R}iemannian geometry},
  volume~5 of {\em Springer INdAM Ser.}, pages 1--13. Springer, Cham, 2014.

\bibitem{ABL17}
A.~A. Agrachev, F.~Boarotto, and A.~Lerario.
\newblock Homotopically invisible singular curves.
\newblock {\em Calc. Var. Partial Differential Equations}, 56(4):56:105, 2017.

\bibitem{AgraBook}
A.~A. Agrachev and Y.~L. Sachkov.
\newblock {\em Control theory from the geometric viewpoint}, volume~87 of {\em
  Encyclopaedia of Mathematical Sciences}.
\newblock Springer-Verlag, Berlin, 2004.
\newblock Control Theory and Optimization, II.

\bibitem{AS96}
A.~A. Agrachev and A.~V. Sarychev.
\newblock Abnormal sub-{R}iemannian geodesics: {M}orse index and rigidity.
\newblock {\em Ann. Inst. H. Poincar\'e Anal. Non Lin\'eaire}, 13(6):635--690,
  1996.

\bibitem{boarotto2015homotopy}
F.~Boarotto and A.~Lerario.
\newblock Homotopy properties of horizontal path spaces and a theorem of
  {S}erre in subriemannian geometry.
\newblock {\em Comm. Anal. Geom.}, 25(2):269--301, 2017.

\bibitem{BryantHsu}
R.~L. Bryant and L.~Hsu.
\newblock Rigidity of integral curves of rank {$2$} distributions.
\newblock {\em Invent. Math.}, 114(2):435--461, 1993.

\bibitem{CJT06}
Y.~Chitour, F.~Jean, and E.~Tr\'elat.
\newblock Genericity results for singular curves.
\newblock {\em J. Differential Geom.}, 73(1):45--73, 2006.

\bibitem{Ger}
V.~Gershkovich.
\newblock On simplest {E}ngel structures on {$4$}-manifolds.
\newblock In {\em Dynamical systems and applications}, volume~4 of {\em World
  Sci. Ser. Appl. Anal.}, pages 279--294. World Sci. Publ., River Edge, NJ,
  1995.

\bibitem{HestQ}
M.~R. Hestenes.
\newblock Applications of the theory of quadratic forms in {H}ilbert space to
  the calculus of variations.
\newblock {\em Pacific J. Math.}, 1:525--581, 1951.

\bibitem{HirschLacombe}
F.~Hirsch and G.~Lacombe.
\newblock {\em Elements of functional analysis}, volume 192 of {\em Graduate
  Texts in Mathematics}.
\newblock Springer-Verlag, New York, 1999.
\newblock Translated from the 1997 French original by Silvio Levy.

\bibitem{Hsu92}
L.~Hsu.
\newblock Calculus of variations via the {G}riffiths formalism.
\newblock {\em J. Differential Geom.}, 36(3):551--589, 1992.

\bibitem{Kato}
T.~Kato.
\newblock {\em Perturbation theory for linear operators}.
\newblock Classics in Mathematics. Springer-Verlag, Berlin, 1995.
\newblock Reprint of the 1980 edition.

\bibitem{SLShortest}
W.~Liu and H.~J. Sussmann.
\newblock Shortest paths for sub-{R}iemannian metrics on rank-two
  distributions.
\newblock {\em Mem. Amer. Math. Soc.}, 118(564):x+104, 1995.

\bibitem{AbnMin}
R.~Montgomery.
\newblock Abnormal minimizers.
\newblock {\em SIAM J. Control Optim.}, 32(6):1605--1620, 1994.

\bibitem{Mont02}
R.~Montgomery.
\newblock {\em A tour of subriemannian geometries, their geodesics and
  applications}, volume~91 of {\em Mathematical Surveys and Monographs}.
\newblock American Mathematical Society, Providence, RI, 2002.

\bibitem{Moser}
J.~Moser.
\newblock On the volume elements on a manifold.
\newblock {\em Trans. Amer. Math. Soc.}, 120:286--294, 1965.

\bibitem{Pontryagin}
L.~S. Pontryagin, V.~G. Boltyanskii, R.~V. Gamkrelidze, and E.~F. Mishchenko.
\newblock {\em The mathematical theory of optimal processes}.
\newblock Translated by D. E. Brown. A Pergamon Press Book. The Macmillan Co.,
  New York, 1964.

\bibitem{Rif14}
L.~Rifford.
\newblock {\em Sub-{R}iemannian geometry and optimal transport}.
\newblock SpringerBriefs in Mathematics. Springer, Cham, 2014.

\bibitem{Sarysec}
A.~V. Sary\v{c}ev.
\newblock Index of second variation of a control system.
\newblock {\em Mat. Sb. (N.S.)}, 113(155)(3(11)):464--486, 496, 1980.

\bibitem{SaryStab}
A.~V. Sary\v{c}ev.
\newblock Stability of mappings of {H}ilbert space and the equivalence of
  control systems.
\newblock {\em Mat. Sb. (N.S.)}, 113(155)(1(9)):146--160, 176, 1980.

\bibitem{Sus4d}
H.~J. Sussmann.
\newblock A cornucopia of four-dimensional abnormal sub-{R}iemannian
  minimizers.
\newblock In {\em Sub-{R}iemannian geometry}, volume 144 of {\em Progr. Math.},
  pages 341--364. Birkh\"auser, Basel, 1996.

\bibitem{Tre00}
E.~Tr\'elat.
\newblock Some properties of the value function and its level sets for affine
  control systems with quadratic cost.
\newblock {\em J. Dynam. Control Systems}, 6(4):511--541, 2000.

\bibitem{Varberg}
D.~E. Varberg.
\newblock On absolutely continuous functions.
\newblock {\em Amer. Math. Monthly}, 72:831--841, 1965.

\end{thebibliography}
	
\end{document}